 \documentclass[11pt,letterpaper]{amsart}

 \usepackage{txfonts} 

\usepackage[colorlinks,linkcolor=blue,citecolor=blue]{hyperref}

 \theoremstyle{plain}
 
\newtheorem{theorem}{Theorem}[section]
\newtheorem{lemma}[theorem]{Lemma}

\theoremstyle{definition}
\newtheorem{definition}[theorem]{Definition}

\newtheorem{corollary}[theorem]{Corollary}
\newtheorem{proposition}[theorem]{Proposition}

\newtheorem{problem}[theorem]{Problem}
 
 \newcommand{\bj}{\bar{j}}
 
 \newcommand{\bl}{\bar{l}}

 \newcommand{\Ric}{\mbox{Ric}}

\theoremstyle{remark}
\newtheorem{remark}[theorem]{Remark}

\numberwithin{equation}{section}

\begin{document}

\title[Fully nonlinear   equations and prescribed curvature problems]
  {Fully nonlinear elliptic equations for some prescribed curvature problems on Hermitian manifolds}


\author{Rirong Yuan }
\address{School of Mathematics, South China University of Technology, Guangzhou 510641, China}
\email{yuanrr@scut.edu.cn}


 


\dedicatory{}

\begin{abstract}
We study fully nonlinear elliptic equations on Hermitian manifolds through blow-up argument and partial uniform ellipticity.  We apply our results to draw geometric conclusions on finding conformal Hermitian metrics with prescribed Chern-Ricci curvature functions. 
By some obstruction from geometric function theory, our assumptions are almost sharp.
\end{abstract}

\maketitle



  \section{Introduction}
  
   
  Let $(M,\omega)$ be a Hermitian manifold of complex dimension $n\geq 2$ with K\"ahler form 
  $\omega=\sqrt{-1}g_{i\bar j}dz_i\wedge d\bar z_j$.
 Under the Chern connection $\nabla$
  the curvature 
  of $\omega$ is locally 
  given by 
  \begin{equation}
  	\label{gqy-pre215}
  	R_{i\bar j k\bar l}=-\partial_{\bar j}\partial_i g_{k\bar l}
  	+g^{p\bar q}\partial_i g_{k\bar q}\partial_{\bar j}g_{p\bar l}.
  	 \nonumber
  \end{equation}
  The Ricci curvature on  K\"ahler manifolds has been well studied
  in huge  
  literature, among which
  \cite{CDS,Cheng1980Yau,Tian1990,Tian2015,Tian-Yau1}, to name just a few,
  starting at least from the  milestone
   work  
   of Aubin \cite{Aubin78} and Yau 
   \cite{Yau78} 
   on Calabi's conjectures and K\"ahler-Einstein metric. 
  Unlike the K\"ahler metric case, there are different Ricci curvatures for non-K\"ahler metric 
  \[ R^{(1)}_{i\bar j} = g^{k\bar l} R_{i\bar j k\bar l}, \; 
  R^{(2)}_{i\bar j} = g^{k\bl} R_{k\bar l i\bar j}, \; 
  R^{(3)}_{i\bar j} = g^{k\bar l} R_{i\bar l k\bar j}, \; 
  R^{(4)}_{i\bar j} = g^{k\bar l} R_{k\bar j i\bar l},\]
  where  $\{g^{i\bar j}\} = \{g_{i\bar j}\}^{-1}$.
  Following \cite{LY2017} we call, for $k = 1,2, 3, 4$, 
  \[ \begin{aligned}
  	\Ric_\omega^{(k)}  = \sqrt{-1} R^{(k)}_{i\bar j} dz_i \wedge d\bar z_{j}
  \end{aligned} \]
  the 
  $k$-th \textit{Chern-Ricci form}. 
The Chern-scalar curvature is given  by ${R}_{{\omega}}=\mathrm{tr}(\omega^{-1}\Ric^{(1)})$.
  The first and second Chern-Ricci curvatures 
  are of particular  importance and deeply connected to the complex geometric structure. 
 The Calabi-Yau theorem  was extended by Tosatti-Weinkove \cite{TW10} to non-K\"ahler case for  first Chern-Ricci form.
  The Hermitian curvature flow related to second Chern-Ricci form was proposed  
  by Streets-Tian \cite{Steets2011Tian} 
 as an important analogue of Ricci flow for Hermitian geometry.
  The third and fourth Chern-Ricci forms were considered by Liu-Yang \cite{LY2017}  who studied relations and geometric properties of Ricci curvatures with respect to different 
   (Levi-Civita, Chern and Bismut) 
  connections.
In \cite{GQY2018}   Guan-Qiu-Yuan 
 studied the conformal deformation of the
  \textit{mixed Chern-Ricci form},
the geometric quantity as a combination of Chern-Ricci forms
  \[\Ric_\omega^{\langle\alpha, \beta, \gamma\rangle} 
  := \alpha \Ric_\omega^{(1)} + \beta \Ric_\omega^{(2)} + \gamma (\Ric_\omega^{(3)} + \Ric_\omega^{(4)}).\]

   This paper is devoted to  looking for 
  conformal 
  metrics with further conditions on Chern-Ricci curvatures. 
   The special case of deforming to constant
   Chern-scalar curvature  is 
  referred  to as   
Chern-Yamabe problem  
proposed by    
\cite{Angella-Calamai-Spotti2007}. 
More general prescribed Chern-scalar curvature problem   was  
further studied  in   \cite{Fusi2022,Ho2021,Yu2023}.
 %
%
 
In this paper, as a special case of our results, we
obtain some conclusion regarding to the Chern-Yamabe 
 problem 
for complete
 noncompact 
 manifolds. 

\begin{theorem}
	\label{thm1-complete}  
Let $(M,\omega)$ be a complete noncompact Hermitian manifold of  nonpositive  Chern-scalar curvature.  
	In addition, we assume
	$R_\omega \leq -\delta$ in $M\setminus K_0$
for some compact subset $K_0$ and 
positive constant $\delta$.	Then  there exists  a unique maximal smooth 
	complete  metric $\tilde{\omega}=e^{u}\omega$ with
	${R}_{\tilde{\omega}}  \equiv -1.$  
\end{theorem}


   This is a complex analogue of   \cite{Aviles1988McOwen2}.  
In fact, we consider more general problems.

\begin{problem}
	\label{Q0}
	In the conformal class of Hermitian metrics,
does there exist  a {\em compact} or {\em complete} metric with prescribed first Chern-Ricci curvature function.
	
\end{problem}

\begin{problem}
	\label{Q1}
	In the conformal class of 
	Hermitian metrics,
	can we find a {\em compact} or {\em complete} metric
	so that it has prescribed mixed Chern-Ricci curvature function.
	
	
\end{problem}

As suggested by 
\cite{CNS3},  
we assume that the curvature function 
 $f$ is 
a \textit{smooth},   \textit{symmetric}, \textit{concave} function defined in $\Gamma\subset\mathbb{R}^n$, where $\Gamma$ is an \textit{open},  \textit{symmetric}, \textit{convex} cone with  vertex at  origin, $\partial \Gamma\neq \emptyset$,  and
$\Gamma_n:=\left\{(\lambda_1,\cdots,\lambda_n)\in \mathbb{R}^n: \forall\lambda_i>0\right\}\subseteq\Gamma.$ 
Following \cite{CNS3},  $\Gamma$ is of {\em type 1}
 if $(0,\cdots,0,1)\in\partial\Gamma$; 
otherwise it is of {\em type 2}.

 In   \cite{Wu2023Zhang}, Wu-Zhang studied 
  prescribed Chern-scalar curvature problem for suitable noncompact manifolds
and obtained  conformal metric, possibly not complete, prescribing  
nonpositive and nonzero  Chern-scalar curvature.  
In \cite{GQY2018}, 
Guan-Qiu-Yuan  considered  Problem \ref{Q1}  in special case and obtained    metric              with prescribing boundary metric.  
Definitely not too surprisingly, things become  more subtle and the problems are rarely known  in the case 
when the resulting metric is complete. 

 In this paper we 
 consider  the problems above 
 through  
fully nonlinear equations 
 \begin{equation} \label{equ1-hehe}	f(\lambda(\omega^{-1}(\chi+\sqrt{-1}\partial\overline{\partial}u)))=
 	\psi e^{\Lambda_0 u},
 \end{equation}
 where  $\Lambda_0>0$ is a constant,  
 $\chi$ is a smooth real $(1,1)$-form, and
  $\lambda(\omega^{-1}(\chi+\sqrt{-1}\partial\overline{\partial}u))=(
 \lambda_1,\cdots,\lambda_n)$ denotes the $n$-tuple of 
  eigenvalues of $\chi+\sqrt{-1}\partial\overline{\partial}u$ with respect to $ {\omega}$.
In addition, $f$ 
is supposed to satisfy the following basic assumptions:
\begin{equation}
	\label{homogeneous-1-buchong2}
	\begin{aligned}
		f >0 \mbox{ in }  \Gamma, \,\, f  =0 \mbox{ on } \partial\Gamma,
	\end{aligned}
\end{equation}
\begin{equation}
	\label{homogeneous-1}
	\begin{aligned}
		f(t\lambda)=t^\varsigma f(\lambda), \mbox{ } \forall \lambda\in\Gamma, \,  t>0,
		\mbox{  for some constant } 0<\varsigma \leq1.
	\end{aligned}
\end{equation}   
When $\Gamma$ is of type 1,
 we additionally assume  that \begin{equation}
 	\label{unbounded-1}
 	\begin{aligned}
 		\lim_{t\rightarrow+\infty}  f(\lambda_1,\cdots,\lambda_{n-1},\lambda_n+t)=+\infty, \,\, \forall \lambda=(\lambda_1,\cdots,\lambda_n)\in\Gamma.
 	\end{aligned}
 \end{equation}  

 We emphasize that  throughout this paper   
 the  following condition is  not required
 \begin{equation}
 	\label{elliptic}
 	\begin{aligned}
 		f_i(\lambda)=f_{\lambda_i}(\lambda):
 		 =\frac{\partial f}{\partial\lambda_i}(\lambda)>0 \mbox{ in } \Gamma, \mbox{ } \forall 1\leq i\leq n.
 	\end{aligned}
 \end{equation}   
 See Lemmas \ref{lemma1-unbound-yield-elliptic}, \ref{lemma5.11}  and  
 \ref{lemma23}.
This is in contrast with  huge literature on second order
 fully nonlinear equations of elliptic and parabolic type.
 
 As is well known, 
conditions  
 \eqref{homogeneous-1-buchong2}-\eqref{unbounded-1} 
 allow the important case: 
 $f=  \sigma_k^{1/k}, \mbox{ } \Gamma=\Gamma_k, $ 
 where 
 $\sigma_k$ is the $k$-th elementary symmetric function,  
 $\Gamma_k$ is  the $k$-th G{\aa}rding cone. In particular, when $k=n$ the equation \eqref{equ1-hehe} is the complex Monge-Amp\`ere equation,
  which is closely related the K\"ahler-Einstein metrics on closed K\"ahler manifolds with negative first Chern class; 
 see  \cite{Aubin78,Yau78}.

	


\begin{definition}
	\label{def1-admissible}

	For the equation \eqref{equ1-hehe}, 
	we say that $u$ is an \textit{admissible}  function if   \begin{equation}\begin{aligned}\lambda(\omega^{-1}(\chi+\sqrt{-1}\partial\overline{\partial}u))\in\Gamma \mbox{ in  }\bar M \, (= M\cup\partial M).\nonumber
	\end{aligned}    \end{equation}  
		Here  $M$ stands for the interior of $\bar M$, $\partial M$ denotes the boundary of $M$. 
	Similarly, 
	we call $u$ a \textit{pseudo-admissible} function if
	\begin{equation}\begin{aligned}
			\lambda(\omega^{-1}(\chi+\sqrt{-1}\partial\overline{\partial}u)) \in \bar{\Gamma}:=\Gamma\cup\partial\Gamma \mbox{ in } 
		 M. \nonumber \end{aligned}    \end{equation} 
	Meanwhile, 
	$u$ is the \textit{maximal} 
	 solution to
	\eqref{equ1-hehe},  
	if
	$u\geq w$ 
	for any admissible solution $w$. 
	Similarly, we have  analogous
	 notions of admissible, pseudo-admissible and maximal conformal metrics, respectively.
	
\end{definition}

\begin{definition}
	\label{def1-admissible-boundary}

	We say that  $\partial M$ is $\Gamma_\infty$-\textit{admissible} for  \eqref{equ1-hehe} if $(\kappa_1,\cdots,\kappa_{n-1})\in \Gamma_\infty,$ where   
	$\kappa_{1}, \cdots, \kappa_{n-1}$ are the eigenvalues of Levi form  $L_{\partial M}$
	of  boundary 	 with respect to $\omega^{\prime}=\left.\omega\right|_{T_{\partial M} \cap J T_{\partial M}}$, and $J$ is the complex structure. 
	Henceforth $$\Gamma_\infty:=\{\lambda'=(\lambda_{1}, \cdots, \lambda_{n-1}):   (\lambda_{1}, \cdots, \lambda_{n-1},R) \in \Gamma \mbox{ for some } R>0\}.$$

\end{definition}

In order to solve   \eqref{equ1-hehe}, the    first %
challenge 
is to derive gradient estimate.
The 
direct
proof of gradient estimate was settled  for fairly restrictive cases in literature
\cite{Blocki09gradient,Guan2010Li,Guan2015Sun,Hanani1996,yuan2018CJM,yuan2021cvpde,ZhangXW},
but cases beyond this are mostly open.  
  Blow-up argument 
  offers another  
  approach to prove gradient estimate as shown by \cite{Chen} for complex Monge-Amp\`ere equation,  by \cite{Dinew2017Kolo}  for complex $k$-Hessian equations with the aid of
  second   estimate in
  \cite{HouMaWu2010}; see also \cite{Gabor} for more general equation satisfying 
  \begin{equation}	
  	\label{addistruc}	
  	\begin{aligned}
\mbox{For any $\sigma<\sup_\Gamma f$ and $\lambda\in\Gamma,$ we have }
 	\lim_{t\rightarrow+\infty} f(t\lambda)>\sigma.
  	\end{aligned}
  \end{equation}

 In  this paper,  
we 
 employ such a contradiction method  to 
set up gradient estimate. 
To this end, we derive the quantitative boundary estimate 
\begin{equation}
	\label{quantitative-BE}
	\begin{aligned} 	\sup_{\partial M} |\partial\overline{\partial} u|\leq C(1+\sup_M|\partial u|^2)  	\end{aligned} \end{equation}
for Dirichlet problem,
adapting
 some idea from prequels
\cite{yuan2017,yuan-regular-DP} (see also the 
 subsequent paper  
 \cite{yuan2020PAMQ}). 
As usual  the 
local barrier technique in \cite{Guan1993Boundary,Guan1998The} (further 
refined
by \cite{Guan12a}) is a key ingredient.
There are more related work  \cite{Boucksom2012,Phong-Sturm2010,Collins2019Picard} 
on Dirichlet problem for complex Monge-Amp\`ere equation 
and complex $k$-Hessian equations, 
in which their method relies specifically on the structure of the
operators, 
which cannot be adopted to treat 
general equations.  
%
 
We obtain existence result, assuming admissible function  instead of subsolution.
 \begin{theorem}
 	\label{thm1-dirichlet}
 	Suppose $(f,\Gamma)$ satisfies \eqref{homogeneous-1-buchong2} and \eqref{unbounded-1}.
 	Let $(\bar M,\omega)$ be a compact Hermitian manifold with smooth $\Gamma_\infty$-admissible boundary. 
 	In addition,
 we assume that 
 $\bar M$ carries a $C^2$-smooth admissible function.
  Then for any  
 $\varphi\in C^\infty(\partial M)$  and $0<\psi\in C^\infty(\bar M)$, there is a unique admissible solution to \eqref{equ1-hehe} with $u=\varphi$ on $\partial M$.
 
 	\end{theorem}


 The case $\Gamma_\infty=\mathbb{R}^{n-1}$ is of interest   
  since the  boundary
  is  automatically 
  $\Gamma_\infty$-admissible, without any geometric condition.
Significantly, in this case we   employ certain Morse function  to construct  admissible functions and then solve 
 Dirichlet problem 
without extra assumption on  $\partial M$,
beyond $\partial M\in C^\infty$. 
This is a fully nonlinear analogue of existence theorem for Poisson's equation and Liouville's equation. 
See Theorem \ref{thm2-existence-bdy-UE}. 
In addition,  
with the aid of some results on 
partial uniform ellipticity
and  singular Yamabe problem,  
we can derive interior estimates   
and 
solve the Dirichlet problem 
with infinite boundary data; 
 see  Section \ref{Dirichlet-problem-3} for more results. 

    \begin{theorem}
  	\label{theorem3-complete}
  		Let $(\bar M,\omega)$ be a compact Hermitian manifold with smooth boundary.   In addition to 
  	 \eqref{homogeneous-1-buchong2} 	 and $\sup_\Gamma f=+\infty$, we assume    $\Gamma_\infty=\mathbb{R}^{n-1}$.
  		  Then for any $0<\psi\in C^\infty(\bar M)$, 
     the equation \eqref{equ1-hehe} possesses  an    admissible solution $u\in C^\infty(M)$ with  
  		   $\underset{z\to\partial M}\lim u(z)=+\infty$.
  	Moreover,   $u$ is minimal in the sense that
  	 $u\leq w$  
  	 for any  admissible 
  	 solution  $w$ 
  	 with infinity boundary data.
  	\end{theorem} 
 
Below we give some obstruction 
to indicate that  
in Theorem \ref{theorem3-complete}  the assumption $\Gamma_\infty=\mathbb{R}^{n-1}$   cannot be dropped in general. 
To do this, we 
define
 the integer   
 for $\Gamma$
\begin{equation}\label{def1-kappa-gamma}
	\begin{aligned}	 {\kappa}_{\Gamma}:=\max \left\{k: ({\overbrace{0,\cdots,0}^{k-\mathrm{entries}}},	{\overbrace{1,\cdots, 1}^{(n-k)-\mathrm{entries}}})\in \Gamma \right\}  \end{aligned}\end{equation} 
  in an attempt 
to connect  admissible function to  notion of $q$-plurisubharmonic function 
 in  several complex variables.   
The constant  $\kappa_\Gamma$ 
was  
  introduced 
 by  
  \cite{yuan2020conformal,yuan-PUE1}
to measure how close  the operator  
 $f$ 
can come to uniform ellipticity 
  (Lemma \ref{yuan-k+1}). 


\begin{remark} [Obstruction]
	\label{remark1-obstruction}
	
		
		
		Let  $\Omega_0, \Omega_1, \cdots, \Omega_m$ be smooth bounded domains in $\mathbb{C}^n$ with $\bar \Omega_i\subset\subset\Omega_0$ and $\bar \Omega_i$ being pairwise disjoint, for all $1\leq i\leq m$. 
		Pick $$(\Omega,\omega)=(\Omega_0  \setminus(\cup_{i=1}^m \bar \Omega_i), 
		\sqrt{-1} \sum  dz_i\wedge d\bar z_i).$$		
Assume 
that
  the following  problem
		   (with  $\Gamma$ being of type 1, that is
		     $0\leq\kappa_\Gamma\leq n-2$) 
		\begin{equation}  f(\lambda(\sqrt{-1}\partial\overline{\partial}u))= \psi e^{\Lambda_0 u}	\mbox{ in }\Omega, \,\, u=+\infty \mbox{ on } \partial\Omega  \nonumber 	\end{equation}
		admits a $C^2$-smooth  admissible 
		solution.
	 Then $u$
  is  a  $\kappa_\Gamma$-plurisubharmonic exhaustion function for $\Omega$.   
This yields that 
$\Omega$ is  Levi 
$\kappa_\Gamma$-pseudoconvex 
by   
 Eastwood-Suria \cite{Eastwood1980Suria} and Suria \cite{Suria-q-convex} 
 (see Theorem \ref{thm1-Eastwood-Suria} below), 
 which 
 contradicts to the shape of $\Omega$.

		
\end{remark}

  \begin{remark}

 
 
  The role of ${\kappa}_{\Gamma}$ becomes apparent as it  serves as a bridge  between  two concepts. On one hand, it links the ideas of admissible function and partial uniform ellipticity  
  within the realm of 
   fully nonlinear PDEs. On the other hand, it establishes a connection with the concepts of $q$-plurisubharmonic function, Levi $q$-pseudoconvexity, and $q$-completeness in several complex variables.
   
 
     
  \end{remark}
 
Building on Theorem \ref{theorem3-complete}  
we prove   
 that 
all of geometric and analytic obstructions to 
solvability of  \eqref{equ1-hehe}  
 on  complete noncompact  
 manifolds  are embodied in  
 asymptotic condition at infinity:
$M$ carries  a  \textit{pseudo-admissible} function  
 satisfying 
\begin{equation}
	\label{asymptotic-assumption1}
	\begin{aligned}
		f(\lambda(\omega^{-1}(\chi+\sqrt{-1}\partial\overline{\partial}
		\underline{v}))) \geq  \psi  e^{\Lambda_0 \underline{v}}
		\,  \mbox{ in } M\setminus K_0, \,\, \underline{v}\in C^2(M),
	\end{aligned}
\end{equation}
where $K_0$ is a compact subset of $M$.  
 This is  a sufficient and necessary condition.

\begin{theorem}
	\label{theorem4-complete-noncompact}
	Let $(M,\omega)$ be a complete noncompact 
	Hermitian manifold. 
	Suppose, in addition to 
 \eqref{homogeneous-1-buchong2} and
  $\sup_\Gamma f=+\infty $,   that  	$\Gamma_\infty=\mathbb{R}^{n-1}$. 
	Given a prescribed function $0<\psi\in C^\infty(M)$ satisfying 
	 \eqref{asymptotic-assumption1}, there is a unique maximal smooth admissible solution $u$ to \eqref{equ1-hehe}  with $u\geq \underline{v}-C_0$ in $M$ for some constant $C_0$.
	
\end{theorem}

 The paper is organized as follows. 
In Section \ref{Dirichlet-problem-4} we first draw some  geometric conclusions related to Problems \ref{Q0} and \ref{Q1}.  
  In Section \ref{preliminaries1} we  summarize some useful results.
Based on partial uniform ellipticity,  in Section 
\ref{sec2-PUE-application} we prove that
  \eqref{equ-deform-Ric+} can be reduced to a   fully nonlinear  equation of elliptic or  uniformly elliptic type.  Furthermore, we construct various type 2 cones, which allows us to study Problem   \ref{Q1} and more general equations with Laplacian terms. 
    In Section \ref{construction} we construct admissible functions using certain Morse functions.
  In Section  \ref{Dirichlet-problem} we solve the Dirichlet problem.
  In Section \ref{Dirichlet-problem-3}  we solve the Dirichlet problem 
  with infinite boundary value condition.  Moreover, we verify the completeness of the obtained metric.
    Under an appropriate asymptotic condition  at infinity, in Section \ref{sec1-noncompactcomplete} we prove the existence 
     of maximal solution to equations 
      on complete noncompact   manifolds. 
The proofs of a priori estimates 
are left to Sections \ref{Sec-Estimates1} and \ref{Sec-Estimates2}. 
 In appendices \ref{appendix2} and \ref{appendix1}  
 we   give the proofs of   Lemmas 
  \ref{yuan-k+1}, \ref{lemma5.11}, \ref{lemma1-unbound-type2},  \ref{lemma23} and   \ref{yuan's-quantitative-lemma}, 
  respectively.

   The author  is  
   indebted to  Professor Yi Liu for 
  answering questions related to the proof of Lemma \ref{lemma-diff-topologuy}. He
 also would like to thank Ze Zhou for helpful discussion on homogeneity lemma.


 \medskip

\section{Geometric conclusions on Problems \ref{Q0} and \ref{Q1}}
\label{Dirichlet-problem-4}



\subsection{Geometric conclusions related to Problem  \ref{Q0}}
We draw some results on Problem  \ref{Q0} by solving the equation
\begin{equation}
	\label{equ-deform-1Ric}
	\begin{aligned}
		f(\lambda(- \tilde{\omega}^{-1}
		\Ric_{\tilde{\omega}}^{(1)}))=\psi,  \,\, \lambda(-\tilde{\omega}^{-1}\Ric^{(1)}_{\tilde{\omega}})\in\Gamma,
		\,\,\tilde{\omega}=e^u\omega.
	\end{aligned}
\end{equation} 

The results are stated as follows.
\begin{theorem}
	\label{thm1-1Ric}
	Suppose $(f,\Gamma)$ satisfies 
	\eqref{homogeneous-1-buchong2}, \eqref{homogeneous-1}  and \eqref{unbounded-1}.  
	Assume that $(\bar M, \omega)$ is a compact Hermitian manifold with smooth $\Gamma_\infty$-admissible boundary 
	and suppose a $C^2$ conformal   metric $\underline{\omega}$ satisfying $\lambda(-\omega^{-1}Ric^{(1)}_{\underline{\omega}})\in\Gamma$. 
	Then for any smooth metric $h$ on $\partial M$ which is conformal to the restriction of $\omega$ to $\partial M$ and $0<\psi\in C^\infty(\bar M)$, there exists a unique smooth metric $\tilde{\omega}=e^u\omega$ 
	satisfying \eqref{equ-deform-1Ric}  
	and $\tilde{\omega}\big|_{\partial M}=h$.

\end{theorem}

Furthermore, we can construct complete metrics when $\Gamma$ is of type 2. 
(The obstruction in Remark \ref{remark1-obstruction} indicates that 
such an assumption is generally necessary).
\begin{theorem}
	\label{theorem4-complete-general} 
	Let $(\bar M, \omega)$ be a compact Hermitian manifold with smooth boundary.
	In addition to  
	\eqref{homogeneous-1-buchong2} and \eqref{homogeneous-1}, we assume $\Gamma_\infty=\mathbb{R}^{n-1}$.
	For any $0<\psi\in C^\infty(\bar M)$,  
	the interior 
	$M$ admits a smooth 
	complete metric $\tilde{\omega}=e^u\omega$  
	satisfying \eqref{equ-deform-1Ric}.
	
\end{theorem}

When 
$(M,\omega)$ is complete and noncompact,  
we solve \eqref{equ-deform-1Ric} under the asymptotic condition at infinity: 
There is a compact set $K_0$ and a positive constant $\Lambda$ such that
\begin{equation}
	\label{asymptotic-condition2}
	\begin{aligned} 
		f(\lambda(- {\omega}^{-1}
		\Ric_{{\omega}}^{(1)})) \geq \Lambda\psi 
		\mbox{ in } M\setminus K_0.
	\end{aligned}
\end{equation}   

\begin{theorem}
	\label{theorem3-complete-general} 
	Suppose, in addition to \eqref{homogeneous-1-buchong2}, \eqref{homogeneous-1} and $\Gamma_\infty=\mathbb{R}^{n-1}$, that $(M,\omega)$ is a complete noncompact Hermitian manifold with pseudo-admissible metric subject to 
	\eqref{asymptotic-condition2} for some $0<\psi\in C^\infty(M)$.
	Then 
	\eqref{equ-deform-1Ric}   is uniquely solvable  in the conformal class of maximal smooth 
	complete  metrics.
\end{theorem}


\begin{proof}
	[Proof of Theorems  \ref{thm1-1Ric}, \ref{theorem4-complete-general} and \ref{theorem3-complete-general}]
	When $f$ satisfies \eqref{homogeneous-1}, 
	 \eqref{equ-deform-1Ric} is reduced to 
	\begin{equation}
		\label{equ-deform-1Ric-2}
		\begin{aligned}
			f(\lambda({\omega}^{-1}(\sqrt{-1} \partial \overline{\partial}u		- {n^{-1}}\Ric^{(1)}_{\omega} )))= \psi   e^{\varsigma (u-\log n)}. \nonumber
		\end{aligned}
	\end{equation} 
	Consequently, by Theorems \ref{thm1-dirichlet} and \ref{theorem4-complete-noncompact} we obtain	Theorems  \ref{thm1-1Ric} and \ref{theorem3-complete-general}, respectively.
Combining  Theorems \ref{theorem3-complete} and  \ref{theorem2-complete-2}, we get Theorem \ref{theorem4-complete-general}.
\end{proof}


\subsection{Geometric conclusions related to Problem  \ref{Q1}}

Below we focus on Problem \ref{Q1} by finding the metric  $\tilde{\omega}=e^{u}\omega$ with 
\begin{equation}
	\label{equ-deform-Ric+}
	\begin{aligned}
		f(\lambda(- \tilde{\omega}^{-1}
		\Ric_{\tilde{\omega}}^{\langle\alpha, \beta, \gamma\rangle}))=\psi,	
		\,\, \beta>0, \,\, \lambda(- {\omega}^{-1}\Ric^{\langle\alpha, \beta, \gamma\rangle}_{\tilde{\omega}})\in\Gamma.	 
	\end{aligned}
\end{equation}   
%
Under the assumption \eqref{homogeneous-1},
the equation \eqref{equ-deform-Ric+}  
reads as follows:
\begin{equation}
	\label{equ-deform-Ric+2}
	\begin{aligned}
		f(\lambda({\omega}^{-1}(\Delta u   \omega 
		+\beta^{-1}(n\alpha+2\gamma)\sqrt{-1} \partial \overline{\partial}u
		- {\beta}^{-1}\Ric^{\langle\alpha, \beta, \gamma\rangle}_{\omega})))= 
		\psi e^{\varsigma (u-\log\beta)}, 
	\end{aligned}
\end{equation} 
where $\Delta u
= \mathrm{tr}({\omega}^{-1}\sqrt{-1}\partial\overline{\partial}u)
$,
according to the formula 
 (see e.g. \cite{GQY2018})  
\begin{equation} \label{conformal-formula1}	\begin{aligned} 	-\Ric^{\langle\alpha, \beta, \gamma\rangle}_{\tilde{\omega}}  	=  \beta \Delta u  \omega +   (n \alpha + 2 \gamma) \sqrt{-1} \partial \overline{\partial}u	-\Ric^{\langle\alpha, \beta, \gamma\rangle}_{\omega}. \nonumber	\end{aligned} \end{equation}  

Given an admissible subsolution,
the Dirichlet problem  
was  solved   by  
\cite{GQY2018}
when
\begin{equation} 	\label{GQY-condition1} 	\begin{aligned}  	\beta+n\alpha+2\gamma>0, \,\, \beta>0, \,\, n\alpha+2\gamma\neq 0,  \end{aligned} 	\end{equation}
under which  the equation  
becomes uniformly elliptic. 
Nevertheless,  
such uniform ellipticity  
possibly breaks down in the case   
\begin{equation}
	\label{critical-1}
	\begin{aligned}
		\beta+n\alpha+2\gamma=0, \,\,  \beta>0.
	\end{aligned}
\end{equation} 
This includes among others the $(n-1)$ Monge-Amp\`ere  equation 
\begin{equation}
	\label{MA1-ricci}
	\begin{aligned}
		( \Delta u \omega 
		- \sqrt{-1} \partial \overline{\partial}u
		- {\beta}^{-1}\Ric^{\langle\alpha, \beta, \gamma\rangle}_{\omega} )^n=e^{nu+\phi}  \omega^n,
	\end{aligned}
\end{equation}
which is exactly 
of $(n-1)$-uniform ellipticity in the sense of Definition \ref{Def1-PUE}.
This poses a challenge,
especially when the resulting metric is required to be 
complete. 

Our strategy is based on partial uniform ellipticity. 
For $\Gamma$, as in \cite{yuan-PUE2-note} we define 
\begin{equation}	\label{def1-varrho-Gamma}	\begin{aligned}	(1,\cdots,1,1-\varrho_\Gamma)\in\partial \Gamma.  \end{aligned}	\end{equation}  
It is easy to see  $1\leq \varrho_\Gamma\leq n.$  In addition, 
 %
	  $\varrho_{\Gamma_k}=\frac{n}{k}. $ In particular, 
	$\varrho_\Gamma=1 \Leftrightarrow \Gamma=\Gamma_n$, and 
	 $\varrho_{\Gamma}=n \Leftrightarrow  \Gamma=\Gamma_1$.

In Proposition \ref{proposition-n-varrho} we prove that
\eqref{equ-deform-Ric+2} is uniformly elliptic under the assumption 
\begin{equation}
\label{assumption1-parameters}
\begin{aligned} 
	\varrho_{\Gamma} 	\beta + {n\alpha+2\gamma} >0,	\mbox{ } \beta>0, 
	\mbox{ }
	n\alpha+2\gamma\neq 0. 
\end{aligned}
\end{equation} 
This condition is in effect sharp.  
As a result, we obtain 

\begin{theorem}
\label{theorem1-complete-general}

Let $(\bar M,\omega)$ be a compact 
Hermitian manifold with smooth boundary.  		Suppose 
\eqref{homogeneous-1-buchong2} and \eqref{homogeneous-1} hold.
For any 
$(\alpha,\beta,\gamma)$  obeying \eqref{assumption1-parameters} and 
$0<\psi\in C^\infty(\bar M)$,
 the interior 
$M$ admits a   smooth 
complete conformal metric $\tilde{\omega}$ satisfying \eqref{equ-deform-Ric+}.

\end{theorem}

Notice in 
the above theorem that
we don't impose subsolution assumption and  the resulting metric is complete, thereby answering some problems left open by 
\cite{GQY2018}. 
%
In addition, we can treat the problem on the complete noncompact manifold with a pseudo-admissible metric satisfying an asymptotic condition.

\begin{theorem}
\label{theorem2-complete-general}
Assume 
\eqref{homogeneous-1-buchong2}, \eqref{homogeneous-1}, \eqref{assumption1-parameters} hold. 
Suppose that $(M,\omega)$ is a complete noncompact Hermitian manifold with pseudo-admissible metric subject to 
\begin{equation}
	\label{asymptotic-condition1}
	\begin{aligned} 
		f(\lambda(- {\omega}^{-1}
		\Ric_{{\omega}}^{\langle\alpha, \beta, \gamma\rangle})) \geq \Lambda\psi 
		\mbox{ in } M\setminus K_0 
	\end{aligned}
\end{equation}  
for some $0<\psi\in C^\infty(M)$ and positive constant $\Lambda$.
Then  there is a unique  smooth maximal 	complete metric satisfying	\eqref{equ-deform-Ric+}. 
\end{theorem}



The obstruction presented in Remark \ref{remark1-obstruction} indicates that in general one could not expect that Theorem  \ref{theorem1-complete-general} 
holds in the limiting case 
\begin{equation}
\label{assumption2-parameters}
\begin{aligned}
	\varrho_{\Gamma} \beta + {n\alpha+2\gamma} =0, 
	\,\,
	\beta>0.
\end{aligned}
\end{equation}
Fortunately, 
we can solve the Dirichlet problem. 
For our purpose, we assume	
\begin{equation}	\label{unbounded-4} \begin{aligned} 	
	\lim_{t\to+\infty}f(\lambda +t(1,\cdots,1,1-\varrho_\Gamma))=+\infty, \,\, \forall \lambda =(\lambda_1,\cdots,\lambda_n)
	\in\Gamma,	\end{aligned} 	\end{equation}
\begin{equation}
\label{admissible-boundary4}
\begin{aligned}
	\sum_{i=1}^{n-1} \kappa_{i} \vec{\bf 1} - \varrho_\Gamma (\kappa_1,\cdots,\kappa_{n-1},0)+t(1,\cdots,1,1- \varrho_\Gamma) \in \Gamma \mbox{ in } \partial M, \mbox{ for } t\gg1,
\end{aligned}
\end{equation}
where and hereafter $\vec{\bf 1}:=(1,\cdots, 1)\in \mathbb{R}^n.$
%
In particular, when $\Gamma=\Gamma_n$ 
we  verify that
\begin{itemize}

\item 
\eqref{assumption2-parameters} reduces to \eqref{critical-1}, 
and then \eqref{equ-deform-Ric+2} reads an $(n-1)$-type  equation. 


\item  \eqref{unbounded-4}  
allows  $f=({\sigma_n}/{\sigma_{k}})^{1/(n-k)}$ with $0\leq k\leq n-2.$

\item  	\eqref{admissible-boundary4} holds 
if and only if  $\kappa_1+\cdots+\kappa_{n-1}>0.$ 

\end{itemize} 

\begin{theorem}
\label{thm2-compact}

Suppose ${\Gamma} \neq \Gamma_1$ and   $(f,\Gamma)$ satisfies 
\eqref{homogeneous-1-buchong2}, \eqref{homogeneous-1},
\eqref{unbounded-4}. 
Let $(\bar M,\omega)$ be a compact Hermitian manifold with smooth   boundary satisfying  \eqref{admissible-boundary4}. 
Given $(\alpha,\beta,\gamma)$  obeying 
\eqref{assumption2-parameters}, assume 
$\lambda(-\omega^{-1}\Ric^{\langle\alpha, \beta, \gamma\rangle}_\omega)\in\Gamma$ in $\bar M$. 
Then for any smooth metric $h$ on $\partial M$ 
which is conformal to the restriction of $\omega$ to $\partial M$, and  $0<\psi\in C^\infty(\bar M)$,  there is a unique smooth metric $\tilde{\omega}=e^u\omega$ satisfying \eqref{equ-deform-Ric+}  
and $\tilde{\omega}\big|_{\partial M}=h$.
\end{theorem}

Finally, we will complete the proof of Theorems \ref{theorem1-complete-general}, \ref{theorem2-complete-general} and  \ref{thm2-compact}.
\begin{proof} [Proof of  Theorems \ref{theorem1-complete-general}, \ref{theorem2-complete-general} and  \ref{thm2-compact}]
	We show in Section 
	\ref{sec2-PUE-application}
	that 
	\eqref{equ-deform-Ric+2} falls into an equation of the form   \eqref{equ1-hehe}.
The equation is of uniform elliptictity under assumption \eqref{assumption1-parameters}  by Proposition 
	\ref{proposition-n-varrho}. 
Therefore, Theorems \ref{theorem1-complete-general} and \ref{theorem2-complete-general} follows from  Theorems 
\ref{theorem3-complete},  \ref{theorem2-complete-2} and \ref{theorem4-complete-noncompact}.    
The equation is elliptic when \eqref{assumption2-parameters} holds according to Proposition \ref{key-lemma2} and Lemma  \ref{lemma2-unbound-condition}. Together with
Lemma \ref{lemma1-boundary-admissible}, we can confirm all the assumptioms in Theorem \ref{thm1-dirichlet}.
	Thus we obtain Theorem \ref{thm2-compact}.
	
	
\end{proof}

\begin{remark}
The Yamabe problem for complete noncompact Riemannian manifolds is not always solvable due to the counterexample of Jin \cite{Jin1988}.
	We reasonably believe that  
	the asymptotic assumptions  
	at infinity   in Theorems  \ref{thm1-complete}, 
	\ref{theorem3-complete-general},
	 \ref{theorem2-complete-general} 
	 and \ref{theorem4-complete-noncompact}
	can not be further dropped  in general.  
	(Also note  that such asymptotic conditions are  
 necessary for the solvability of the equations, 	respectively).
\end{remark}


  
   \section{Preliminaries}
  \label{preliminaries1}

	Throughout this paper, 
 $\sigma(z)$ denotes the distance from $z$ to $\partial M$, and $f$ satisfies the natural condition 
 $\sup_{\partial\Gamma} f< \sup_{\Gamma} f,$
where 
 $ \sup_{\partial\Gamma} f= \sup_{\lambda_0\in\partial\Gamma}  {\limsup}_{\lambda\to\lambda_0}f(\lambda).$ 
 In  computation we use derivatives with respect to the Chern connection $\nabla$ of $\omega$,
and write
$\partial_{i}=\frac{\partial}{\partial z_{i}}$, 
$\overline{\partial}_{i}=\frac{\partial}{\partial \bar z_{i}}$,
$\nabla_{i}=\nabla_{\frac{\partial}{\partial z_{i}}}$,
$\nabla_{\bar i}=\nabla_{\frac{\partial}{\partial \bar z_{i}}}$.
For   a smooth function $v$,
\begin{equation} \label{formula-1}\begin{aligned}
		v_i:=
		\partial_i v,  
		\mbox{  } v_{\bar i}:=\,&
		\partial_{\bar i} v,
		\,
		v_{i\bar j}:=  
		\partial_i\overline{\partial}_j v, 
		\,
		v_{ij}:=
		\partial_j \partial_i v -\Gamma^k_{ji}v_k,  \\
		v_{i\bar j k} :=\,&\partial_k v_{i\bar j} -\Gamma_{ki}^l v_{l\bar j},  \cdots, \mbox{etc}, 
\end{aligned}\end{equation} 
where $\Gamma_{ij}^k$ are the Christoffel symbols
defined  by 
$\nabla_{\frac{\partial}{\partial z_i}} \frac{\partial}{\partial z_j}=\Gamma_{ij}^k \frac{\partial}{\partial z_k}.$ 

For simplicity, we denote   
\begin{equation}
	\begin{aligned}  
		\psi[u]  =  \psi(z,u), \nonumber
		\,\,
		\lambda(\Omega)=\lambda(\omega^{-1}\Omega) \mbox{ for real $(1,1)$-form } \Omega.  \nonumber
	\end{aligned}
\end{equation}      
\begin{align*}
	\partial\Gamma^\sigma=\{\lambda\in\Gamma: f(\lambda)=\sigma\},\, \Gamma^\sigma=\{\lambda\in\Gamma: f(\lambda)>\sigma\}.
	\end{align*}

  \subsection{Some result on Morse function}
  
 The following lemma asserts that any compact manifold with boundary carries some function without any critical points.    

  \begin{lemma}
  	\label{lemma-diff-topologuy}
  	Let 
  	$\bar M$
   be a compact connected 
  	manifold of dimension $n\geq 2$ with smooth boundary. Then there is a smooth function $v$ without any critical points. 
  \end{lemma}

\begin{proof} 

	 The construction  is more or less standard in differential topology. 
	Let $X$ be the double of $M$. Let $w$ be a smooth Morse function on $X$ with the critical set $\{p_i\}_{i=1}^{m+k}$, among which $p_1,\cdots, p_m$ are all the critical points  being in $\bar M$. 
	Pick $q_1, \cdots, q_m\in X\setminus \bar M$ but not the critical point of $w$. By homogeneity lemma 
	(see \cite{Milnor-1997}), 
	one can find a diffeomorphism
	$h: X\to X$, which is smoothly isotopic to the identity, such that  
 $h(p_i)=q_i$ for $1\leq  i\leq m$, and moreover 	$h(p_i)=p_i$ for $m+1\leq  i\leq  m+k$.
	Then $v=w\circ h^{-1}\big|_{\bar M}$ is the desired 
	function.
	
\end{proof}

  
  \subsection{The criterion of $q$-pseudoconvexity} 
  
First,  we recall some related notion.

\begin{definition}
	An open set $\Omega\subset\mathbb{C}^n$ is called Levi $q$-pseudoconvex if 
	at  any $x\in\partial \Omega$ 
	the Levi form $L_\rho$
	has at most $q$-negative eigenvalues on the holomorphhic tangent space ${T_x,}_{\partial\Omega}\cap J{T_x,}_{\partial\Omega}.$
\end{definition}

  \begin{definition} 
	A $C^2$ function  $h: \Omega\to\mathbb{R}$ is called a  $q$-plurisubharmonic function if $\sqrt{-1}\partial\overline{\partial} h$ has at least $n-q$
	positive eigenvalues for all $z$ in $\Omega$.
	
\end{definition}
 
According to some  results of   Eastwood-Suria \cite{Eastwood1980Suria} and  Suria \cite{Suria-q-convex},  one has 
  \begin{theorem}
  [\cite{Eastwood1980Suria,Suria-q-convex}] 
  	\label{thm1-Eastwood-Suria}
  	 Let $\Omega\subset\mathbb{C}^n$ be a $C^2$-smoothly bounded domain.
  	  Then $\Omega$ is Levi-$q$ pseudoconvex if and only if it admits a $C^2$-smooth $q$-plurisubharmonic exhausion function.
  \end{theorem}
	
For more results on $q$-pseudoconvexity,  $q$-complete and $q$-plurisubharmonic function, please refer to the monograph  \cite{Ohsawa2022Pawlaschyk}.  
Also, 
we 
refer to \cite{Cheng1980Yau,Mok-Yau1983} for some results relating 
Ricci curvature to function-theoretic information.

\subsection{A quantitative lemma}


The following lemma was proposed by 
\cite{yuan2017}\renewcommand{\thefootnote}{\fnsymbol{footnote}}\footnote{The results in \cite{yuan2017}   
	were moved to \cite{yuan-regular-DP}. More precisely, the paper  \cite{yuan-regular-DP} is essentially extracted from 
	\cite{yuan2017},  	along with the first parts of   [arXiv:2001.09238] and   [arXiv:2106.14837].}, 
which can be  
 viewed 
 as a quantitative version of \cite[Lemma 1.2]{CNS3}.
  For completeness, we  present the proof in Appendix \ref{appendix1}.  
\begin{lemma}[\cite{yuan2017,yuan-regular-DP}]
	\label{yuan's-quantitative-lemma}
	Let $A$ be an $n\times n$ Hermitian matrix
	\begin{equation}\label{matrix3}\left(\begin{matrix}
			d_1&&  &&a_{1}\\ &d_2&& &a_2\\&&\ddots&&\vdots \\ && &  d_{n-1}& a_{n-1}\\
			\bar a_1&\bar a_2&\cdots& \bar a_{n-1}& \mathrm{{\bf a}} \nonumber
		\end{matrix}\right)\end{equation}
	with $d_1,\cdots, d_{n-1}, a_1,\cdots, a_{n-1}$ fixed, and with $\mathrm{{\bf a}}$ variable.
	Denote the eigenvalues of $A$ by $\lambda=(\lambda_1,\cdots, \lambda_n)$.
	Let $\epsilon>0$ be a fixed constant.
	Suppose that  
	\begin{equation}
		\begin{aligned}
			\label{guanjian1-yuan}
			\mathrm{{\bf a}}\geq \frac{2n-3}{\epsilon}\sum_{i=1}^{n-1}|a_i|^2 +(n-1)\sum_{i=1}^{n-1} |d_i|+ \frac{(n-2)\epsilon}{2n-3}. \nonumber
		\end{aligned}
	\end{equation}
	Then the eigenvalues 
	(possibly with a proper permutation)
	behave like
	\begin{equation}
		\begin{aligned}
			d_{\alpha}-\epsilon 	\,& < 
			\lambda_{\alpha} < d_{\alpha}+\epsilon, \mbox{  } \forall 1\leq \alpha\leq n-1, \\ \nonumber
			\mathrm{{\bf a}} 	\,& \leq \lambda_{n}
			< \mathrm{{\bf a}}+(n-1)\epsilon. \nonumber
		\end{aligned}
	\end{equation}
\end{lemma}

\subsection{Useful lemmas regarding to $f$} 
\label{lemmas-for-f}

The concavity  of $f$ 
yields a useful inequality 
\begin{equation}
	\label{concavity2}
	f(\mu) \leq f(\lambda)+	\sum_{i=1}^n	f_i(\lambda)  (\mu_i-\lambda_i), \, \forall \lambda,\, \mu\in\Gamma.
\end{equation}
 
The following lemma states that the unbound and concavity   imply  monotonicity.
This was observed in 
new 
draft 
of \cite{yuan-regular-DP}.

\begin{lemma}  
	\label{lemma1-unbound-yield-elliptic}
	If $f$ satisfies 
	\eqref{unbounded-1} in $\Gamma$,
	then   \eqref{elliptic} holds.
	
\end{lemma}
\begin{proof}
	Suppose $\lambda_1\leq\cdots\leq\lambda_n$. Then $f_1(\lambda)\geq\cdots\geq f_n(\lambda)$.
	In view of the concavity and unbound of $f$, by setting $t\gg1$  we know 
	\begin{equation}
		\begin{aligned}
			f_n(\lambda)\geq\frac{f(\lambda_1,\cdots,\lambda_{n-1},\lambda_n+t)-f(\lambda)}{t}>0. \nonumber
		\end{aligned}
	\end{equation} 
\end{proof}

Next, we present another useful lemma. 

\begin{lemma}[\cite{yuan-regular-DP,yuan-PUE1}]
	\label{lemma3.4}
	
	If $f$ satisfies \eqref{addistruc}, then  	$\sum_{i=1}^n f_i(\lambda)\mu_i>0$, $\forall \lambda,\mbox{ }  \mu\in\Gamma.$
\end{lemma}

  		\begin{proof}	
	Let  $\sigma=f(\lambda)$ and 
	$Df =(f_1,\cdots,f_n)$. 
	By \eqref{addistruc}, $t\mu\in\Gamma^\sigma$  for $t$ large. 
	Since  $\Gamma^\sigma$ is convex, $Df(\lambda)\cdot (t\mu-\lambda)>0$. So $Df(\lambda)\cdot\lambda> 0$ (setting $\mu=\lambda$)
	and so $Df(\lambda)\cdot\mu>0$. 
	

\end{proof}


We introduce the following notion in order to explore the structure of  fully nonlinear equations of elliptic and parabolic type.
\begin{definition}
	 [Partial uniform ellipticity]
	\label{Def1-PUE}
	We say that
	$f$ is of \textit{$\mathrm{m}$-uniform ellipticity} in $\Gamma$,  
	if  $f$ satisfies  	
	\begin{enumerate}
		
		\item[$\mathrm{(i)}$]
		$f_i(\lambda)\geq 0,   \mbox{ } \forall 1\leq i\leq n, \mbox{  }
		\sum_{i=1}^n f_i(\lambda)>0, \mbox{   } \forall \lambda\in\Gamma. $
		
		\item[$\mathrm{(ii)}$] There is a uniform 
		constant $\vartheta$ 
		such that for $\lambda\in \Gamma$ with 
		$f_1(\lambda)\geq \cdots\geq f_n(\lambda)$,
		\begin{equation}
			\label{partial-uniform2}
			\begin{aligned}
				f_{i}(\lambda)\geq \vartheta\sum_{j=1}^n f_j(\lambda)>0, \mbox{ } \forall 1\leq i\leq \mathrm{m}. 
			\end{aligned}
		\end{equation}
		  
	\end{enumerate}
	In particular, \textit{$n$-uniform ellipticity} is also called	\textit{fully uniform ellipticity}.
	
	Accordingly, we have an analogous notion of partial uniform ellipticity for a second order elliptic equation, 	if its linearized operator satisfies a similar condition.
\end{definition}

  In \cite{yuan2020conformal,yuan-PUE1}\renewcommand{\thefootnote}{\fnsymbol{footnote}}\footnote{The results in \cite{yuan2020conformal} were moved to  \cite{yuan-PUE1}.} 
 the author  determined the integer $\mathrm{m}$ from \eqref{partial-uniform2} for generic  symmetric concave
  functions,
 extending extensively an inequality of 
 \cite{Lin1994Trudinger}  for $f=\sigma_k^{1/k}$. 


\begin{lemma}
	 [\cite{yuan2020conformal,yuan-PUE1}] 
	\label{yuan-k+1}
	Suppose \eqref{addistruc} holds. Then for any 
	$\lambda\in \Gamma$ with
	$\lambda_1 \leq \cdots \leq\lambda_n$, we have
	\begin{enumerate}
		\item   $f_i(\lambda)\geq 0,   \mbox{ } \forall 1\leq i\leq n, $ 
		$\sum_{i=1}^n f_i(\lambda)>0.$
		\item  $f_{{i}}(\lambda) 
		\geq n    \vartheta_{\Gamma}f_1(\lambda) 
		\geq \vartheta_{\Gamma} \sum_{j=1}^{n}f_j(\lambda),  \mbox{ } \forall  1\leq i\leq 1+ 	\kappa_{\Gamma}.$
	\end{enumerate}
	Here $\kappa_\Gamma$ is given in \eqref{def1-kappa-gamma}, and
	\begin{equation}
		\label{theta-gamma}
		\vartheta_\Gamma=
		\begin{cases}  
			1/n, \,& \Gamma=\Gamma_n,\\
			\underset{(-\alpha_1,\cdots,-\alpha_{\kappa_\Gamma}, \alpha_{\kappa_\Gamma+1},\cdots, \alpha_n)\in \Gamma,\mbox{ } \alpha_i>0}{\sup}\frac{\alpha_1/n}{\sum_{i=\kappa_\Gamma+1}^n \alpha_i-\sum_{i=2}^{\kappa_\Gamma}\alpha_i}, \,& \Gamma\neq\Gamma_n.  
		\end{cases}  
	\end{equation}
 Moreover, the assertion of $(\kappa_\Gamma+1)$-uniform ellipticity  
 is sharp.
\end{lemma} 

 \begin{remark}	If $\Gamma=\Gamma_1$ then by \eqref{theta-gamma} $f_1(\lambda)=f_2(\lambda)=\cdots=f_n(\lambda)$ in $\Gamma.$\end{remark}






It is notable that $f$ is of 
uniform ellipticity when $\Gamma$ is of type 2  and vice versa.
\begin{lemma} 
	 [\cite{yuan2020conformal,yuan-PUE1}]
	\label{lemma5.11}
	Suppose $f$ satisfies  
	\eqref{addistruc} in $\Gamma$.  
	Then the following 
	are equivalent: 
	\begin{enumerate}
		
			 		\item $\Gamma_\infty=\mathbb{R}^{n-1}$.
			 		
	 		\item $\kappa_\Gamma=n-1$. That is,  $\Gamma$ is of type 2. 

		\item  
	There exists a uniform constant $\theta$ such that
		\begin{equation}
			\label{uniform-elliptic-2}
			\begin{aligned}
				f_i(\lambda)\geq \theta \sum_{j=1}^n f_j(\lambda)>0      
				\mbox{ in } \Gamma,
				\,  \forall 1\leq i\leq n.
			\end{aligned}
		\end{equation}

	\end{enumerate}
\end{lemma}

When  
 $f$ satisfies 	 \eqref{uniform-elliptic-2}, 
we verify the unbounded  condition. 	 
\begin{lemma}
	[\cite{yuan-PUE1}]
	\label{lemma1-unbound-type2}
	In the presence of    \eqref{addistruc},  
	 \eqref{uniform-elliptic-2}  
	and $\sup_\Gamma f=+\infty$, 
$f$ satisfies the unbounded condition	\eqref{unbounded-1}.
\end{lemma}



Finally, we verify \eqref{addistruc} in certain case.
\begin{lemma}[\cite{yuan-PUE1}]
	\label{lemma23}
	Suppose  $\sup_\Gamma f=+\infty$ and \eqref{homogeneous-1-buchong2} holds. Then 
	$f$ satisfies \eqref{addistruc}.
\end{lemma}

For convenience,   in Appendix \ref{appendix2} we will give  the proofs of above lemmas.

  
 \medskip

  \section{On the structure of fully nonlinear equations}
  \label{sec2-PUE-application}

In this section we explore the structure of  fully nonlinear equations of the type
\begin{equation}
	\label{equation-n-varrho}
	\begin{aligned}
		f(\lambda(\chi+\Delta u \, \omega-\varrho\sqrt{-1}\partial\overline{\partial}u))=\psi[u], 
		 \,\, \varrho\neq0,
	\end{aligned}
\end{equation}
 and  the relation to
 \begin{equation} \label{equ3-hehe}	\tilde{f}(\lambda( \tilde{\chi}+\sqrt{-1}\partial\overline{\partial}u))= \psi[u].
\end{equation}


\subsection{The 
	relation between \eqref{equation-n-varrho} and  \eqref{equ3-hehe}}
\label{subsec1-equi}

We prove that \eqref{equation-n-varrho} can be transformed into \eqref{equ3-hehe} and vice versa.  
Let $\varrho_\Gamma$ be   
as  defined in  \eqref{def1-varrho-Gamma}.

\subsubsection{Equation \eqref{equation-n-varrho} has the form \eqref{equ3-hehe}}

Fix $(f,\Gamma)$.
 Given 
$\varrho$ with
 $\varrho\leq\varrho_\Gamma$, $\varrho\neq0$ 
($\varrho<n$ if $\Gamma=\Gamma_1$),   
 we  can construct 
$(\tilde{f},\tilde{\Gamma})$ as follows:
%
%
\begin{equation}
	\label{map1}
	\begin{aligned}
		\tilde{\Gamma}  =
		\left\{(\lambda_1,\cdots,\lambda_n): 
		\lambda_i=\frac{1}{\varrho}\left(\sum_{j=1}^n \mu_j-(n-\varrho)\mu_i\right),
		\mbox{ } (\mu_1,\cdots,\mu_n)\in \Gamma 
		\right\}.
	\end{aligned}
\end{equation}  
Note that for any $\lambda\in\tilde{\Gamma}$   there is a unique $\mu\in\Gamma$ such that
$\lambda_i=\frac{1}{\varrho}\left(\sum_{j=1}^n \mu_j-(n-\varrho)\mu_i\right).$
Define $\tilde{f}: \tilde{\Gamma}\to \mathbb{R}$ by
\begin{equation}\label{def-f}\begin{aligned}  \tilde{f}(\lambda)=f(\mu). 
	\end{aligned}\end{equation}
This shows that   \eqref{equation-n-varrho} (with $\varrho\leq\varrho_\Gamma$) has the form \eqref{equ3-hehe}.
One can simply verify that $\tilde{f}$ is 
 concave in $\tilde{\Gamma}$.
Furthermore, if $(f,\Gamma)$ satisfies \eqref{addistruc} then so does $(\tilde{f},\tilde{\Gamma})$.


\subsubsection{Equation  \eqref{equ3-hehe} has the form  \eqref{equation-n-varrho}}

\begin{proposition}
	\label{proposition2-Gamma}
Suppose as before, $\tilde{\Gamma}$ is an open symmetric  convex cone with vertex at origin, and with $\partial\tilde{\Gamma}\neq\emptyset$,   $\tilde{\Gamma}\neq\Gamma_1$, $\Gamma_n\subseteq\tilde{\Gamma}$. Pick a constant $\varrho$ with
	$n-\varrho_{\tilde{\Gamma}}\leq \varrho <n$.
	Let $P$ be a linear map   from $\mathbb{R}^n$  to $\mathbb{R}^n$ which is defined as follows:
	\begin{equation}
		\label{form1-type2cone}
		\begin{aligned}
			P(\lambda)= \frac{1}{n-\varrho}(\sum_{j=1}^n \lambda_j \vec{\bf 1}-\varrho\lambda). \nonumber
		\end{aligned}
	\end{equation}
Take $\Gamma:=P(\tilde{\Gamma}). $
	Then $\Gamma$ is an open, symmetric,  convex cone with vertex at origin, 
	 $$ \Gamma_n\subseteq\Gamma\subset \Gamma_1, \,\, \varrho\leq\varrho_{\Gamma}.$$
\end{proposition}

\begin{proof}
	Since   $\tilde{\Gamma}\neq\Gamma_1$, 
	 $0<n-\varrho_{\tilde{\Gamma}}\leq\varrho<n$ and  the linear map $P$ is well-defined and invertible. Moreover,  
	  $\Gamma$ is an open, convex, symmetric cone with vertex at origin.
 
	Fix $\lambda\in\tilde{\Gamma}$, we take $\mu=P(\lambda)$.  
	By $\sum_{j=1}^n \mu_j= \sum_{j=1}^n\lambda_j$,  we know $\Gamma\subseteq\Gamma_1.$ 
The remaining step is to prove $\Gamma_n\subseteq\Gamma.$
Since $n-\varrho_{\tilde{\Gamma}}\leq\varrho$, we know $(1,\cdots,1,1-n+\varrho)\in\overline{\tilde{\Gamma}},$ which implies $(0,\cdots,0,1)\in \bar   \Gamma$. This means $\Gamma_n\subseteq\Gamma.$ 
Since $(0,\cdots,0,1)\in\overline{\tilde{\Gamma}}$, we get $(1,\cdots,1,1-\varrho)\in\bar\Gamma$. Thus $\varrho\leq \varrho_\Gamma.$
	
\end{proof}

Consequently, $\tilde{\Gamma}$  has the form  \eqref{map1}. 
  In other words, for any 
  cone $\tilde{\Gamma}$ (with $\tilde{\Gamma}\neq\Gamma_1$) as in Proposition \ref{proposition2-Gamma},
 there
 is an open symmetric convex cone $\Gamma$ with vertex at origin 
 and a constant $\varrho$ of $0<n-\varrho_{\tilde{\Gamma}}\leq \varrho\leq\varrho_\Gamma$, such that
\begin{equation} 
	\begin{aligned}
		\tilde{\Gamma}  =
		\left\{(\lambda_1,\cdots,\lambda_n): 
		\lambda_i=\frac{1}{\varrho}\left(\sum_{j=1}^n \mu_j-(n-\varrho)\mu_i\right) 
		\mbox{ for } (\mu_1,\cdots,\mu_n)\in \Gamma 
		\right\}.  \nonumber
	\end{aligned}
\end{equation}  
As a result,  
\eqref{equ3-hehe} can be rewritten   in the form  \eqref{equation-n-varrho} (with $0< n-\varrho_{\tilde{\Gamma}}\leq\varrho\leq\varrho_\Gamma$).


\subsection{On the structure of $(\tilde{f},\tilde{\Gamma})$}

We can prove by Lemma \ref{lemma5.11} that
\begin{proposition}
	\label{proposition-n-varrho}
	Suppose $(f,\Gamma)$ satisfies   \eqref{addistruc}.
	Then
	\begin{itemize}
		\item If 
		$\varrho<\varrho_\Gamma$,  $\varrho\neq0$,
		then \eqref{equation-n-varrho}
		is uniformly elliptic at admissible solution $u$ with
		$\lambda(\chi+\Delta u \omega-\varrho\sqrt{-1}\partial\overline{\partial}u)\in\Gamma.$
		\item If $\varrho=\varrho_\Gamma$ ($\Gamma\neq\Gamma_1$) and 
		\eqref{unbounded-4} holds, then \eqref{equation-n-varrho}
		is  elliptic at admissible  solutions.
	\end{itemize}
\end{proposition}


To achieve this  we first check that
\begin{lemma}
	\label{proposition1-newvwesion}
	Given a cone $\Gamma$, as in \eqref{map1} we take $\tilde{\Gamma}$.
	Then 
	\begin{enumerate}
		\item 
		$\tilde{\Gamma}$ is of type 2 
		if and only if 
		$\varrho<\varrho_\Gamma$, $\varrho\neq0.$

		\item 
		$\tilde{\Gamma}$ is of type 1 if $\varrho=\varrho_{\Gamma}.$
	\end{enumerate}

\end{lemma}


\begin{remark}
	 This was  also observed in 
	 \cite{yuan-PUE2-note}.
	A somewhat surprising 
	fact to us is that   $(n-1)$-type fully nonlinear equation  is of uniform ellipticity whenever  $\Gamma\neq\Gamma_n$.
	This is in contrast with  the $(n-1)$ Monge-Amp\`ere equation,
	which is in close connections with $(n-1)$-plurisubharmonic functions in the sense of Harvey-Lawson \cite{Harvey2012Lawson} as well as Form-type Calabi-Yau equation \cite{FuWangWuFormtype2010} and Gauduchon's conjecture \cite{Gauduchon84}  
	(see also
	\cite{Popovici2015,Tosatti2019Weinkove}).  
	In recent years,  Sz\'ekelyhidi-Tosatti-Weinkove \cite{STW17} proved the Gauduchon conjecture  for  
	higher dimensions, 
	extending earlier work of 
	Cherrier  \cite{Cherrier1987} on complex surfaces. Subsequently, the author  \cite{yuan-n-1MA} solved the Dirichlet problem, in which the equation 
	probably allows degeneracy.   
\end{remark}

\subsubsection{Uniform ellipticity case}

In practice,    
Lemma  \ref{lemma5.11}  and the first part of Lemma \ref{proposition1-newvwesion} 
together give  the following key ingredient. 

\begin{proposition} 
	\label{key-lemma1}
Fix   a constant with $\varrho<n$, $\varrho  \neq0$.
Given $(f,\Gamma)$ satisfying  \eqref{addistruc}, 
  as  in \eqref{map1} and \eqref{def-f}	we can define $(\tilde{f},\tilde{\Gamma})$.
	 Then the following are equivalent:
	 \begin{itemize}
	 	\item 
	 		$\varrho<\varrho_\Gamma$.
\item
  $\tilde{f}$ is of fully uniform ellipticity in $\tilde{\Gamma}$. Namely,  
	\begin{equation}	 
		\label{FUE-1}
		\begin{aligned}
			\frac{\partial \tilde{f}}{\partial \lambda_i}(\lambda) \geq \theta \sum_{j=1}^n \frac{\partial \tilde{f}}{\partial \lambda_j}(\lambda) >0
		\mbox{ in }	\tilde{\Gamma},   \mbox{ } \forall 1\leq i\leq n. \nonumber
		\end{aligned}
	\end{equation}
	 \end{itemize}
	\end{proposition}

\subsubsection{Ellipticity case}

Assume $\varrho=\varrho_\Gamma$ and $\Gamma\neq\Gamma_1$. 
Obviously, from the construction of $(\tilde{f},\tilde{\Gamma})$  we have
\begin{lemma}
	\label{lemma2-unbound-condition}
  Given $(f,\Gamma)$ satisfying  \eqref{addistruc},
   we assume $\varrho=\varrho_\Gamma$ ($\Gamma\neq\Gamma_1$).
Let $(\tilde{f},\tilde{\Gamma})$ be as  in \eqref{map1} and \eqref{def-f}.
	Suppose in addition that  $f$ obeys \eqref{unbounded-4} in $\Gamma$. Then  
	$\tilde{f}$ satisfies the unbounded condition \eqref{unbounded-1} in $\tilde{\Gamma}$.
\end{lemma}


Together with Lemma \ref{lemma1-unbound-yield-elliptic} we can conclude that
\begin{proposition}
	\label{key-lemma2}
	Suppose  
	$\varrho=  \varrho_{\Gamma}$ (${\Gamma} \neq \Gamma_1$) and that  $(f,\Gamma)$ satisfies 
	\eqref{addistruc}  
	and \eqref{unbounded-4}.
Let $(\tilde{f},\tilde{\Gamma})$ be as in \eqref{map1} and \eqref{def-f}. Then 
$\tilde{f}$ satisfies \eqref{elliptic} in $\tilde{\Gamma}.$ 
 That is	\begin{equation}	\frac{\partial\tilde{f}}{\partial\lambda_i}(\lambda)>0,  \mbox{ } \forall	\lambda 
 	\in\tilde{\Gamma},  \mbox{ } \forall  1\leq i\leq n.  \nonumber\end{equation}
\end{proposition}





 \subsection{Further remarks on $\Gamma_\infty$-admissible boundary}

 	Let $\tilde{\Gamma}$ be as in \eqref{map1}.
  We can  check that
 \begin{lemma}
 	\label{lemma1-boundary-admissible}
 
 $\partial M$ is $\tilde{\Gamma}_\infty$-admissible  (i.e., $(\kappa_1,\cdots,\kappa_{n-1})\in \tilde{\Gamma}_\infty$) 	if and only if
 	\begin{equation}
 		\label{admissible-boundary3}
 		\begin{aligned}
 			\sum_{i=1}^{n-1} \kappa_{i} \vec{\bf 1} -\varrho (\kappa_1,\cdots,\kappa_{n-1},0)+t(1,\cdots,1,1-\varrho) \in \Gamma \mbox{ for } t\gg1.
 		\end{aligned}
 	\end{equation}
 \end{lemma}
 

 \begin{corollary}
 	If 
 		$\varrho<\varrho_\Gamma$, $\varrho\neq0,$ then 
 		 any smooth boundary is $\tilde{\Gamma}_\infty$-admissible. 
 		 On the other hand, when $\varrho= \varrho_\Gamma$  the condition \eqref{admissible-boundary3} coincides with \eqref{admissible-boundary4}.
 \end{corollary}

 
 	


  \medskip

 \section{Construct  admissible functions via Morse functions}
 \label{construction}

 
 \begin{lemma}
 	\label{lemma-construction-function}
 	Let $(\bar M,\omega)$ be a compact Hermitian manifold 
 	with smooth boundary.
 	Then there is a smooth admissible function $\underline{w}$ subject to $\lambda(\chi+\sqrt{-1}\partial\overline{\partial}  \underline{w})\in\Gamma$ in $\bar M$,
 	provided that $\Gamma$ is of type 2, i.e., $\Gamma_\infty=\mathbb{R}^{n-1}$.
 \end{lemma}
 \begin{proof}
 	By Lemma \ref{lemma-diff-topologuy}, we have a  smooth  function $v$  with  $v\geq1$ and
 	$\partial v\neq 0$ on $\bar M$.  Let
 	$\underline{w}=e^{tv}$. Note that $\lambda(\sqrt{-1}\partial v\wedge\overline{\partial}v)=|\partial v|^2(0,\cdots,0,1)$ and
 	\begin{equation}
 		\label{formular1-compu}
 		\chi+\sqrt{-1}\partial\overline{\partial}  \underline{w}=\chi+te^{tv}(\sqrt{-1}\partial\overline{\partial} v+t\sqrt{-1}\partial v\wedge \overline{\partial} v).
 	\end{equation}
 	Since $\Gamma_\infty=\mathbb{R}^{n-1}$, 
 	$\lambda(\sqrt{-1}\partial\overline{\partial} v+t\sqrt{-1}\partial v\wedge \overline{\partial} v)\in\Gamma$ for $t\gg1$. 
 	Together with the openness of $\Gamma$,  $\underline{w}$ is an  admissible function when $t\gg1$. 
 	
 \end{proof}  	
 
 
 

 

 \begin{lemma}	
 	\label{lemma1-asymcondition}
 	Under the assumptions of Theorem 
 	\ref{theorem4-complete-noncompact}, there exists a $C^2$-admissible function 
 	$\underline{u}$ 
 	satisfying 
 	\begin{equation}
 		\label{asymp-condition2}
 		\begin{aligned}
 			f(\lambda(\chi+\sqrt{-1}\partial\overline{\partial}\underline{u}))\geq \Lambda_1\psi  e^{\Lambda_0\underline{u}} \,\, \mbox{ in } M  
 		\end{aligned}
 	\end{equation}
 	for some constant $\Lambda_1>0$. Moreover, $\underline{u}\geq \underline{v}-C_1$ for some $C_1>0$, 
 	where $\underline{v}$ is as in Theorem \ref{theorem4-complete-noncompact}.
 	
 \end{lemma}
 
 \begin{proof}
 	Without loss of generality,  we may assume that $\underline{v}=0$ is pseudo-admissible and   satisfies 
 	\eqref{asymptotic-assumption1}.
 	Let $K_0$ be   the compact subset as in 
 	\eqref{asymptotic-assumption1}.    	From  \eqref{homogeneous-1-buchong2}, the pseudo-admissible assumption and the positivity of $\psi$, we know 
 	$\lambda(\chi)\in \Gamma$ in $M\setminus K_0$.
 	
 	Pick  
 	compact 
 	submanifolds $M_1$, $M_2$ of complex dimension $n$ and with smooth boundary  and with $K_0\subset\subset M_1\subset\subset M_2$. Choose a cutoff function satisfying
 	\begin{equation}
 		\begin{aligned}
 			\zeta\in C^{\infty}_0(M_2), \, 0\leq\zeta\leq 1 \mbox{ and } \zeta\big|_{M_1}=1. \nonumber
 		\end{aligned}
 	\end{equation}  
 By   Lemma \ref{lemma-diff-topologuy}, we   take  a smooth 
 	function $v$  with $dv\neq0$ and 
 	$v\leq  0$ 
 	on $\bar M_2$.
 	From \eqref{formular1-compu},  for $t\gg1$,	$\underline{w}=e^{t(v-1)}$ 
 	is  an admissible function on $\bar M_2$.  
Take $\underline{u}=e^{Nh}$, where
 	\begin{equation}
 		h=
 		\begin{cases}
 			\zeta  v -1\,& \mbox{ if } x\in M_2,\\
 			-1 \,& \mbox{ otherwise.}\nonumber
 		\end{cases}
 	\end{equation} 
 When $N\gg1$, $\underline{u}$ is an admissible function and satisfies \eqref{asymp-condition2}. 
 	
 \end{proof}

   \medskip
  
  \section{The Dirichlet problem}
  \label{Dirichlet-problem}


 From Subsection \ref{subsec1-equi} we know that  \eqref{equ-deform-Ric+2}, and so \eqref{equ-deform-Ric+},  falls into equation of the form 
\eqref{equ1-hehe}. From now on, we consider  more general equation than  \eqref{equ1-hehe}  
\begin{equation}
	\label{main-equ2}
	\begin{aligned}
		f(\lambda(\omega^{-1}(\chi+\sqrt{-1}\partial\overline{\partial}u)))
		=  \psi(z,u).
	\end{aligned}
\end{equation} 
  
Throughout this section, and Sections \ref{Dirichlet-problem-3} as well as   \ref{Sec-Estimates1},  we suppose that $(\bar M,\omega)$ is a compact Hermitian manifold with smooth boundary.

In this  section we consider 
  the equation  \eqref{main-equ2} 
  prescribing boundary value data
  \begin{equation}\label{bdy-condition2} \begin{aligned}		
  		u=\varphi \,\mbox{ on } \partial M.
  	\end{aligned}
  \end{equation}
  Furthermore, we assume that
  $\psi(z,t)$ is a smooth function on $\bar M\times \mathbb{R}$ with 
  \begin{equation}
  	\label{nondegenerate-assumption-1}
  	\begin{aligned}	
   \inf_{z\in  M}\psi(z,t)>\sup_{\partial\Gamma}f, \,\, \forall t\in\mathbb{R}.
  	  	\end{aligned}
  \end{equation}


\begin{theorem}
	\label{thm2-existence-bdy}
	Let $(\bar M,\omega)$ be a compact Hermitian manifold with smooth $\Gamma_\infty$-admissible boundary.  
	In addition to 
	\eqref{unbounded-1}  and \eqref{addistruc}, we assume that $\psi(z,t)$ is a smooth function on $\bar M\times \mathbb{R}$ subject to  \eqref{nondegenerate-assumption-1} 
	and
		\begin{equation}
		\label{assump2-psi}
		\begin{aligned}
			\psi_t(z,t):=\frac{\partial \psi(z,t)}{\partial t}>0, \mbox{  } \forall 
			(z,t)\in  M\times\mathbb{R},
		\end{aligned}
	\end{equation}
	\begin{equation}
		\label{assump1-psi}
		\begin{aligned}
			\lim_{t\to -\infty}\psi(z,t)= \inf_{\Gamma} f, \mbox{  } \forall 
			z\in\bar M.
		\end{aligned}
	\end{equation}
	Suppose in addition that there is a $C^{2}$ admissible function $\underline{w}$. 
	Then for any $\varphi\in C^\infty(\partial M)$, there is a unique smooth admissible function satisfying \eqref{main-equ2} and \eqref{bdy-condition2}.
	
\end{theorem}



When $\Gamma_\infty=\mathbb{R}^{n-1}$  
or 
$f$ satisfies \eqref{uniform-elliptic-2},   
we will show that the Dirichlet problem is 
uniquely solvable without assumptions on boundary and existence of admissible function, beyond $\partial M\in C^\infty$. 
It is a fully nonlinear analogue of existence theory for Poisson's equation and  Liouville's equation.

\begin{theorem}  
	\label{thm2-existence-bdy-UE}
	Let $(\bar M,\omega)$ be a compact Hermitian manifold with smooth boundary.  Suppose 
	$\Gamma_\infty=\mathbb{R}^{n-1}$, $\sup_\Gamma f=+\infty$
	and that $(f,\Gamma)$  satisfies \eqref{addistruc}. 
 Assume in addition that $\psi(z,t)$ 
	satisfies \eqref{nondegenerate-assumption-1} and 
	$\psi_t(z,t)\geq0.$  
	Then for any   
	$\varphi\in C^\infty(\partial M)$, 
 the Dirichlet problem  
	\eqref{main-equ2} and \eqref{bdy-condition2} has a unique smooth admissible solution.
		
\end{theorem}




As a special case, we obtain 
\begin{theorem}  
	\label{thm1-existence-bdy-UE}
Suppose   	$\Gamma_\infty=\mathbb{R}^{n-1}$,  
	$\sup_\Gamma f=+\infty$ and that $(f,\Gamma)$ satisfies  \eqref{homogeneous-1-buchong2}. Then for any smooth positive function $\psi$ in $\bar M$ and $\varphi\in C^\infty(\partial M)$, the equation \eqref{equ1-hehe} possesses a unique smooth admissible solution  
	with 
	$u=\varphi$ on $\partial M$.
\end{theorem}




\subsection{Set-up}

According to  the
  Evans-Krylov theorem  \cite{Evans82,Krylov83} 
  and Schauder theory, 
   it 
  suffices to 
  establish
   estimates for complex Hessian up to boundary
  \begin{equation}
  	\label{estimate-c0-c2}
  	\sup_{\bar M}|\partial \overline{\partial} u|\leq C. 
  \end{equation}
  
  Let $\hat{u}$ be the solution to
  \begin{equation}
  	\label{supersolution1}
  	\begin{aligned}
  		\Delta \hat{u}+\mathrm{tr}(\omega^{-1}\chi)=0 \mbox{ in } M, \,\, 
  		\hat{u}=\varphi \mbox{ on } \partial M. 
  	\end{aligned}
  \end{equation}
  The existence and regularity 
  of $\hat{u}$ can be found in standard textbooks; see e.g. \cite{GT1983}. The maximum principle yields
  \begin{equation}
  	\label{upper-comparison}
  	\begin{aligned}
  		u\leq \hat{u} \mbox{ in } M, \,\, 
  		u=\hat{u}=\varphi \mbox{ on } \partial M. 
  	\end{aligned}
  \end{equation}

 \noindent \textit{\bf Key assumption}:
  Near the boundary 
  we assume that
  there exists a local admissible function $\underline{u}$ satisfying
  \begin{equation}
  	\label{admifunction1-local}
  	\begin{aligned}
  		u\geq \underline{u}  \mbox{ in } M_\delta, \,\,
  		\underline{u}=\varphi \mbox{ on } \partial M 
  	\end{aligned}
  \end{equation} 
  for  some $\delta>0$,
  where 
    \begin{equation}
  	\label{M-delta}
  	\begin{aligned}
  		M_{\delta}:=\{z\in M: \sigma(z)<\delta\}.
  	\end{aligned}
  \end{equation}

  \begin{lemma}
  	\label{lemma-c0-boundary-c1}
  	Any  admissible solution $u$ satisfying   \eqref{admifunction1-local} shall obey
  	\begin{equation}
  		\label{c0-boundary-c1}
  		\sup_M u\leq C, \,\, \sup_{\partial M}|\partial u|\leq C.
  	\end{equation}
  
  	Moreover, 
  if replacing  local condition \eqref{admifunction1-local} by a global version
  	\begin{equation}	
  		\label{admifunction1}
  		u\geq \underline{u}  \mbox{ in } M, \,\,	\underline{u}=\varphi \mbox{ on } \partial M,
  	\end{equation} 
   then we have zero order and boundary gradient estimates
  	\begin{equation}
  		\label{c0-boundary-c1-2}
  		\sup_M |u|+\sup_{\partial M}|\partial u|\leq C.
  	\end{equation}
  \end{lemma}


  %
  The primary problem is to derive gradient estimate as described in introduction. 
  Our strategy is to establish quantitative boundary estimate of the form \eqref{quantitative-BE}, i.e.,
  \begin{equation}
  	\begin{aligned}
  		\sup_{\partial M} |\partial\overline{\partial} u|\leq C(1+\sup_M|\partial u|^2), \nonumber
  	\end{aligned}
  \end{equation}
given a local admissible function near boundary.
  We leave the proof to Section \ref{Sec-Estimates1}. 
  
  \begin{proposition}
  	\label{bdy-quantive-estimate-1} 
  	Assume 	
  	 \eqref{unbounded-1}, \eqref{addistruc} and 	 \eqref{admifunction1-local} hold.
  	Then for any admissible solution $u\in C^3(M)\cap C^2(\bar M)$ to the Dirichlet problem  \eqref{main-equ2}-\eqref{bdy-condition2}, we have the quantitative boundary estimate \eqref{quantitative-BE}.
  \end{proposition}

   
   On the other hand, 
  following closely the proof of 
 Hou-Ma-Wu \cite[Theorem 1.1]{HouMaWu2010}, or the generalization by Sz\'ekelyhidi \cite[Section 4]{Gabor},
   one can derive 
  \begin{proposition}
  	\label{global-quantive-estimate-1}
  	Suppose,  in addition to
  	\eqref{addistruc}
  	and \eqref{unbounded-1},
  	that there is a  $C^2$-smooth admissible function $\underline{w}$.
  	Then for any admissible solution  $u\in C^4(M)\cap C^2(\bar M)$ to equation \eqref{main-equ2}, there is a uniform constant $C$ such that
  	\begin{equation}
  		\begin{aligned}
  			\sup_{M} |\partial\overline{\partial} u| \leq 
  			C(1+\sup_M|\partial u|^2+\sup_{\partial M} |\partial\overline{\partial} u|). \nonumber
  		\end{aligned}
  	\end{equation}	
  \end{proposition}

  The above two propositions together give
  \begin{equation}
  	\sup_M |\partial\overline{\partial} u| \leq C(1+\sup_M|\partial u|^2). \nonumber
  \end{equation}
Using the Liouville type theorem of  Sz\'ekelyhidi \cite{Gabor}, 
 we can derive gradient estimate and therefore \eqref{estimate-c0-c2}. 

  \subsection{The Dirichlet problem   on manifolds with $\Gamma_\infty$-admissible boundary}
  \label{Dirichlet-problem-2}


  \subsubsection{$C^0$-estimate}
  By  maximum principle,  we obtain the following  estimate 
 as a complement to \eqref{upper-comparison}.    Since the proof is standard, we omit it here.
  \begin{lemma}
  	\label{lemma-c0general}
  	In addition to \eqref{addistruc},
  	 \eqref{assump2-psi}, \eqref{assump1-psi}, we assume that there is an admissible function $\underline{w}$. Let  
  	$u\in C^2(\bar M)$ be an admissible solution to 
  	 \eqref{main-equ2}-\eqref{bdy-condition2},
  	then 
  	\begin{equation}
  		\begin{aligned}
  			\inf_M (u-\underline{w}) \geq 
  			\min\left\{\inf_{\partial M}(\varphi-\underline{w}), \mbox{ } A_1-\sup_M\underline{w}\right\},
  			\nonumber
  		\end{aligned}
  	\end{equation}
  	where  $A_1$ is a constant with
  		$\underset{z\in M}\sup \, \psi(z,A_1)\leq \underset{M}\inf \,  f(\lambda(\chi+
  	\sqrt{-1}\partial\overline{\partial}\underline{w})).$  
  	
  \end{lemma}

\subsubsection{The construction of local barriers}
\label{construction1-local-subsolution}

 As above   $\sigma$ denotes the distance function 
 to $\partial M$,
 and $\kappa_1,\cdots,\kappa_{n-1}$ are the eigenvalues of Levi form $L_{\partial M}$. 
  Under the  assumption $(\kappa_1,\cdots,\kappa_{n-1})\in\Gamma_\infty$, 
  we may use   
  $\sigma$  
  to construct local barriers, thereby confirming  \eqref{admifunction1-local}. 
  Fix $k\geq1$.  
Similar to the Riemannian case (see e.g. \cite{Guan2008IMRN}) we take
  \begin{equation}
  	\label{barrier1-w}
  	\begin{aligned}
  		w(z)= 
  		2\log \frac{\delta^2}{\delta^2 +k\sigma(z)}. 
  	\end{aligned}
  \end{equation}
  
The straightforward computation gives the following:
  \begin{equation}
  	\begin{aligned}
  		\partial\overline{\partial} w
  			=\frac{ 
  			 2k}{\delta^2+k\sigma} \left(\frac{k}{\delta^2+k\sigma}\partial\sigma\wedge\overline{\partial}\sigma-\partial\overline{\partial}\sigma\right). \nonumber
  	\end{aligned}
  \end{equation}

  Note that  $|\partial\sigma| =\frac{1}{2}$ on $\partial M$, and   $\frac{k}{\delta^2+k\sigma}=\frac{1}{\sigma+ {\delta^2}/{k}}\geq \frac{1}{\sigma+ {\delta^2} }$ on    $M_\delta$. 
Together with Lemma \ref{yuan's-quantitative-lemma}, we  can take
 $0<\delta\ll1$ such that $w$ is smooth in $M_\delta$ and
  \begin{equation}
  	\begin{aligned}
  		\lambda\left(
  		\frac{k\sqrt{-1}}{\delta^2+k\sigma}\partial\sigma\wedge\overline{\partial}\sigma
  		-\sqrt{-1}\partial\overline{\partial}\sigma
  		\right)\in\Gamma \mbox{ and }
  		\lambda\left(\chi+\sqrt{-1}\partial\overline{\partial}  \varphi
  		+\frac{\sqrt{-1}}{2}  \partial\overline{\partial}  w
  		\right)\in\Gamma 
  		\mbox{ in } M_{\delta}. \nonumber
  	\end{aligned}
  \end{equation}
  Here is the only place to use the $\Gamma_\infty$-admissible assumption  
  on the boundary. 

By Lemma \ref{lemma23}, $f$ obeys  \eqref{addistruc}. Using lemma    \ref{lemma3.4},  we can derive $f(\lambda+\mu)\geq f(\lambda)$ for $\lambda, \, \mu\in\Gamma$.   Notice  $w\leq 0$.
Take $0<\delta_1\ll1$, we  conclude 
  \begin{equation}
  	\label{mp1-inequality}
  	\begin{aligned}
 f\left(\lambda(\chi+ \sqrt{-1}\partial\overline{\partial}(w+\varphi))\right)
  \geq  
  f(\lambda(\frac{\sqrt{-1}}{2}\partial\overline{\partial} w))   
  		\geq   
  		 \psi  (z,w+\varphi) \mbox{ in } M_{\delta_1}.  
  	\end{aligned}
  \end{equation}
By Lemma \ref{lemma-c0general}, 
$u$ has a uniform lower bound, i.e., there is a  constant $\delta_2$ such that 
 \begin{equation} \label{lower-bound2}	\begin{aligned} \underset{M}\inf(u-\varphi)\geq 
 		2\log\frac{\delta_2}{\delta_2+k}.  	\end{aligned} \end{equation}
Consequently, the comparison principle 
yields that 
    \begin{equation}
  	\label{lower-comparison}
  	\begin{aligned}
  		u  \geq 	w+\varphi=
  		2\log \frac{\delta^2}{\delta^2 +k\sigma}+\varphi \mbox{ on } M_{\delta}, \, \delta=\min\{\delta_1,\delta_2\}.
  	\end{aligned}
  \end{equation}
So
$\underline{u}=w+\varphi$ is a desired local admissible function satisfying  \eqref{admifunction1-local}.

\subsection{The Dirichlet problem with type 2 cone}

To obtain Theorem \ref{thm2-existence-bdy-UE}, it suffices to confirm  
\eqref{admifunction1-local} and 
 \eqref{c0-boundary-c1-2}.

\begin{proposition}  
	\label{thm2-estimate-bdy-UE}
	Let $(\bar M,\omega)$ be a compact Hermitian manifold with smooth boundary.  Suppose 
	$\Gamma_\infty=\mathbb{R}^{n-1}$, $\sup_\Gamma f=+\infty$
	and that $(f,\Gamma)$  satisfies \eqref{addistruc}. 
Assume  $\varphi\in C^\infty(\partial M)$ and $\psi(z,t)$  is a  smooth function
	satisfying  \eqref{nondegenerate-assumption-1} and
	$\psi_t(z,t)\geq0$.
	Let $u\in C^2(\bar M)$ be an admissible solution to the Dirichlet problem 
	\eqref{main-equ2} and \eqref{bdy-condition2}. Then $u$
	satisfies \eqref{admifunction1-local} and \eqref{c0-boundary-c1-2}.
	
\end{proposition}



 \begin{proof} 	
	According to Lemma  \ref{lemma5.11}, 	
	$f$ satisfies \eqref{uniform-elliptic-2}.
In addition, $f$ satisfies 
	\eqref{unbounded-1}  by   Lemma \ref{lemma1-unbound-type2}. 
	%
%
From \eqref{upper-comparison} we know  $u$ has a upper bound $u\leq \hat{u}$. Together with   
	$\psi_u(z,u)\geq0$,
	we know there is a uniform constant $C_1$ such that
	\begin{equation}
		\begin{aligned}
			\psi(z,u)\leq \psi(z,\hat{u})\leq C_1, \, \forall z\in M. \nonumber
		\end{aligned}
	\end{equation}
As in proof of Lemma \ref{lemma-construction-function}, let $\underline{w}=e^{tv}$. Note $\lambda(\sqrt{-1}\partial v\wedge\overline{\partial}v)=|\partial v|^2(0,\cdots,0,1)\in\Gamma$. Then $\underline{w}$ is admissible for $t\gg1$.
	Using  \eqref{formular1-compu}  and Lemma 
for $t\gg1$ we get
	\begin{equation}
		\begin{aligned}
			f(\lambda(\chi+\sqrt{-1}\partial\overline{\partial}\underline{w})) >
			f(\lambda(\chi+\sqrt{-1}\partial\overline{\partial} u)). \nonumber
		\end{aligned}
	\end{equation}
By the maximum principle,  $u$ has a uniform lower bound 
	\begin{equation}
		\label{inequality2}
	\begin{aligned}
		\underset{M} \inf( u-\underline{w})=\underset{\partial M} \inf(\varphi-\underline{w}).	
\end{aligned}
\end{equation}
Hence  \eqref{lower-bound2}. 
As in \eqref{barrier1-w} we take  $w= 
2 \log \frac{\delta^2}{\delta^2 +k\sigma}$.
		Similar to \eqref{mp1-inequality}, we get
	\begin{equation} 
		\begin{aligned}
			f\left(\lambda(\chi+ \sqrt{-1}\partial\overline{\partial}(w+\varphi))\right)
			>C_1\geq 
			f(\lambda(\chi+\sqrt{-1}\partial\overline{\partial} u)) \mbox{ in } M_{\delta_1}. \nonumber
		\end{aligned}
	\end{equation}
	Then we get \eqref{lower-comparison}, hence confirming \eqref{admifunction1-local}. 
	Combining with \eqref{inequality2} we obtain  \eqref{c0-boundary-c1-2}.

 \end{proof}

  \section{The Dirichlet problem with infinite boundary value condition, and completeness of conformal metrics}
  \label{Dirichlet-problem-3}

  


  	When  the right-hand side
  $\psi(z,t)$  satisfies  exponential growth in $t$ at infinity, 
  we can solve the Dirichlet problem with infinity boundary data.
  
  \begin{theorem} 
  	\label{theorem1-complete-2} 
  	Let $(\bar M,\omega)$ be a compact Hermitian manifold with smooth boundary.
  	In addition to 
  	$\Gamma_\infty=\mathbb{R}^{n-1}$, $\sup_\Gamma f=+\infty$, we assume  $f$ satisfies 
  	\eqref{addistruc}. 
Let   $\psi(z,t)$  be a  smooth function
satisfying  \eqref{nondegenerate-assumption-1} and $\psi_t(z,t)\geq0$. Suppose in addition   that 
     	  \begin{equation}
    	\label{assump5-psi} 
     \psi(z,t) \geq h(z)e^{ l(z) t}, 
    	\,\, \forall z\in \bar M, \, \forall t>T
    \end{equation}
    for some  $T>0$ and positive valued continuous
     functions $h$, $l\in C^0(\bar M)$.  
  	There is  
  	an admissible function $u\in C^\infty(M)$ satisfying
  	\eqref{main-equ2} and $\underset{z\to\partial M}\lim u(z)=+\infty.$	
  		Moreover, $u$ is minimal in the sense that $u\leq w$  
  	for any  admissible solution $w$ with infinity boundary value.
  \end{theorem}

When $\psi(z,t)=\psi(z) e^{\Lambda_0 t}$ we obtain Theorem \ref{theorem3-complete}.
Moreover, we have 
 \begin{theorem}
 	\label{theorem2-complete-2}  
 	Let $u$ be the  minimal solution asserted in Theorem \ref{theorem1-complete-2}.
 Suppose  the assumptions in Theorem \ref{theorem3-complete} hold.
In additon, we assume   that $f$  obeys \eqref{homogeneous-1}  and 
$\psi(z,t)=\psi(z) e^{\varsigma t}$. Here $\varsigma$ is as in \eqref{homogeneous-1}.  Then $e^u\omega$ is complete.

 	\end{theorem}
  
  
  
  
  
  \subsection{Lemmas}
   To fix the notation,  $\nabla_{{g}}^2 u$  denotes the real Hessian of $u$ under Levi-Civita connection   of $(M,g)$
   (a Riemmannian manifold of real dimension $2n$). 
 Let  $\Delta_{{g}} u=\mathrm{tr}({g}^{-1}\nabla_{{g}}^2 u)$.
 It is  known that the complex Laplacian 
 differs from
  standard Laplacian of Levi-Civita connection 
 by a linear first order term;  see \cite{Gauduchon84}.  That is
  \begin{lemma}
  	\label{lemma-Gauduchon1}
  Let $(M,\omega)$ be a Hermitian manifold of complex dimension $n$. 
  	Let $\tau$ be the torsion $1$-form with 
  	$
  	d\omega^{n-1}=\omega^{n-1}\wedge \tau. 
  	$
  	For any $u\in C^2(M)$, we have
  	$$2\Delta u=\Delta_{{g}} u-\langle du,\tau\rangle_{\omega}.$$

  	
  \end{lemma}

The following important result is due to Aviles-McOwen \cite{Aviles1988McOwen}, who extended extensively a seminal result  of Loewner-Nirenberg \cite{Loewner1974Nirenberg}.
\begin{lemma}
	[\cite{Aviles1988McOwen}]
	\label{lemma1-Aviles-McOwen}
	
Suppose that $(\bar X,g)$
is a compact Riemannian manifold of real dimension $m\geq3$ with smooth boundary $\partial X$, $\bar X:=X\cup\partial X$. Then the interior $X$ admits a  complete conformal metric  with negative constant scalar curvature.	
\end{lemma}

From 
 \eqref{concavity2} 
 we can deduce the following lemma. 
  \begin{lemma}
  	\label{lemma-add5}
  Assume $f(\vec{\bf1})<\sup_\Gamma f$ and   $A_f= {n}\left(\sum_{i=1}^n f_i(\vec{\bf 1})\right)^{-1}$.
  	Then
  	\begin{equation}
  		\label{key1-main}
  		\begin{aligned}
  			\sum_{i=1}^n \lambda_i\geq n+ A_f \left(f(\lambda)-f(\vec{\bf 1})\right), \mbox{ }\forall \lambda\in\Gamma. \nonumber
  		\end{aligned}
  	\end{equation}
  	
  \end{lemma} 
  
  \subsection{Proof of Theorems \ref{theorem1-complete-2} and \ref{theorem2-complete-2}}
By Lemma  \ref{lemma5.11}, 	 $\Gamma_\infty=\mathbb{R}^{n-1}$ 
 implies that $f$ satisfies \eqref{uniform-elliptic-2} in $\Gamma$.
 According to Theorem \ref{thm2-existence-bdy-UE}, for any integer $k\geq 1$, the following Dirichlet problem has a unique smooth admissible solution 
  \begin{equation}
  	\label{equations-k}
  	\begin{aligned}
  		f(\lambda(\chi+\sqrt{-1}\partial\overline{\partial}u_{(k)}))
  		= \psi (z, u_{(k)})    \mbox{ in } M,  
  		\,\,  u_{(k)} 
  		=	
  		2\log k 
  		 \mbox{ on }   \partial M.
  	\end{aligned}
  \end{equation} 
  The  comparison
  principle (see e.g. \cite{GT1983})
   yields that
  \begin{equation}
  	\label{inequality-key4}
  	\begin{aligned}
  		u_{(k)}\leq u_{(k+1)}  \mbox{ in } M, \,\, \forall k\geq 1.
  	\end{aligned}
  \end{equation}
So $u_{(k)}$ has a common lower bound for all $k\geq1$.
On the other hand,
from the  
assumption
\eqref{assump5-psi} it follows that there are  
positive
constants $\gamma$, $\Lambda$ and $T_1$ such that
\begin{equation}
	\label{define-Lambda}
	\begin{aligned}
		\psi(z,t) \geq \gamma e^{\Lambda t},  \,\, \forall   z\in \bar M, \, \forall t\geq T_1.
	\end{aligned}
\end{equation}
  
Below we prove local $C^0$ bound from above.
Applying Lemma \ref{lemma1-Aviles-McOwen} to  
Hermitian manifold 
$(M,\omega)$
 (note that it is a manifold of real dimension $2n$ with Riemannian  metric $g$),
there is $\tilde{u}\in C^\infty(M)$  with
\begin{equation}
	\label{scalar-equ1}
	\begin{aligned}
		\frac{1}{2}\Delta_{{g}}\tilde{u}
		+ \frac{n-1}{4}| d\tilde{u}|^2_{g}
		-\frac{S_g}{2(2n-1)}= e^{\tilde{u}} \mbox{ in } M, 
		\,\,\lim_{z\rightarrow \partial M} \tilde{u}(z)=+\infty,
	\end{aligned}
\end{equation}
where 
$S_g$  is the Riemannian scalar curvature of $g$. That is,   $\tilde{g}=e^{\tilde{u}}g$ is a complete metric with Riemannian scalar curvature  $S_{\tilde{g}}=-2(2n-1)$.
  %
  Together with  Lemma  \ref{lemma-Gauduchon1},
  we may use Cauchy-Schwarz inequality to 
  verify the following key lemma.
  \begin{lemma}
  	\label{Prop-key2}
  	Let $\tilde{u}$ be as in \eqref{scalar-equ1}. There exists  a uniform constant $A$ such that
  	\begin{equation}
  		\label{AM-equ-1}
  		\begin{aligned}
  			\Delta \tilde{u} \leq  e^{\tilde{u}} +A
  			\mbox{ in } M, 
  			\,\, \lim_{z\rightarrow \partial M} \tilde{u}(z)=+\infty.
  		\end{aligned}
  	\end{equation}
  \end{lemma}
  
  We have the following  proposition. 
  \begin{proposition}
  	\label{lemma-key3}
  	Let $\tilde{u}$ be as in  \eqref{scalar-equ1}.
  	Let $\Lambda$, $T_1$ be as in \eqref{define-Lambda}.
  	There is a uniform constant $C_o$ depending on   $\inf_M \tilde{u}$  and other known data but not on $k$ such that
  	\begin{equation}
  		\begin{aligned}
  			u_{(k)} \leq \max\left\{ \frac{\tilde{u}+ C_o}{\Lambda},\, T_1+\frac{\tilde{u}-\inf_M {\tilde{u}}}{\Lambda} \right\}  \mbox{ in } M, \, \forall k\geq1. \nonumber
  		\end{aligned}
  	\end{equation}
  \end{proposition} 
  \begin{proof}
  	Fix $k$.
By Lemma \ref{lemma-add5}, we get
  	\begin{equation} 
  		\begin{aligned}
  			\Delta u_{(k)}  \geq 
  			A_f \left(\psi [u_{(k)}]-f(\vec{\bf 1})\right)
  			+n	-\mathrm{tr} (\omega^{-1}\chi). \nonumber
  		\end{aligned}
  	\end{equation}
  We know that $\Lambda u_{(k)}-\tilde{u}$   
  attains its maximum at some interior point $x_0$, where
  $\Delta \tilde{u}\geq \Lambda\Delta u_{(k)}.$
  We assume $u_{k}(x_0)\geq T_1$ (otherwise we are done).  Then   at $x_0$ we have 
  $\psi [u_{(k)}]\geq \gamma e^{\Lambda u_{(k)}}$ by \eqref{define-Lambda}, hence
  \begin{equation}
  	\label{equ-111}
  	\begin{aligned}
  		\Delta u_{(k)}  \geq 
  		A_f \left(\gamma e^{\Lambda u_{(k)}}-f(\vec{\bf 1})\right)
  		+n	-\mathrm{tr} (\omega^{-1}\chi).
  	\end{aligned}
  \end{equation}
  	Combining   \eqref{equ-111} and \eqref{AM-equ-1}, we have  
  	\begin{equation}
  		\begin{aligned}
  			\Lambda A_f \gamma e^{\sup_M(\Lambda u_{(k)} -\tilde{u})}
  			<
  			 1+ \sup_M \left[ e^{-\tilde{u}}\left(A+\Lambda A_f f(\vec{\bf 1})+\Lambda\mathrm{tr}  (\omega^{-1}\chi)  
  			 \right) \right],  \nonumber
  		\end{aligned}
  	\end{equation} 
  	where $A$ comes from Lemma \ref{Prop-key2}. 
  	This completes the proof.
 
  \end{proof} 

\begin{proof}
  [Proof of  Theorem \ref{theorem1-complete-2}]
  By \eqref{inequality-key4} and Proposition \ref{lemma-key3}  the following limit 
  exists
  \begin{equation}
  	\label{limt1}
  	\begin{aligned}
  		u(z)=\lim_{k\to +\infty} u_{(k)}(z), 
  		\mbox{  } \forall z\in M.  
  	\end{aligned}
  \end{equation}
  Using the interior estimates 
  proved in 
  Proposition \ref{thm0-inter}, together with  Evans-Krylov theorem and  Schauder theory,  we know
  $u\in C^\infty(M).$ 
  On the other hand, by the maximum principle, 
we have   $u\leq w$ for any admissible solution $w$ with $w\big|_{\partial M}=+\infty$.
  
  \end{proof}


\begin{proof}
  [Proof of Theorem \ref{theorem2-complete-2}]
Note that in Theorem \ref{theorem2-complete-2}, 
   $f$ satisfies \eqref{homogeneous-1} and
  $\psi(z,u)=\psi(z) e^{\varsigma u}$.
In \eqref{barrier1-w} and \eqref{equations-k},  we take 
  $ w= 2\log \frac{\delta^2}{\delta^2 +k\sigma }$, $\varphi=2\log k$. That is 
  \begin{equation}
	\label{equations-k-2}
	\begin{aligned}
		f(\lambda(\chi+\sqrt{-1}\partial\overline{\partial}u_{(k)}))
		= \psi e^{\varsigma u_{(k)}}   \mbox{ in } M,  
		\,\,  u_{(k)}=2\log k  \mbox{ on }   \partial M. \nonumber
	\end{aligned}
\end{equation} 
Let $u$ be the limit as  we defined in \eqref{limt1}. 
Using Lemma \ref{lemma3.4} and $(0,\cdots,0,1)\in\Gamma$, we can
  establish  an inequality similar to \eqref{mp1-inequality}. 
   Notice by \eqref{inequality-key4}   that $u_{(k)}$ has a common lower bound for all $k\geq1$.
  Therefore, by comparison principle shows that there is a uniform 
  constant $\delta$ 
  such that
  \begin{equation}
  	\begin{aligned}
  		u_{(k)} (z)\geq  2\log \frac{k\delta^2}{\delta^2 +k\sigma(z)}  \mbox{ on } M_\delta, \,\, \forall k\geq1.  \nonumber
  	\end{aligned}
  \end{equation}
 Thus $u\geq -2\log\sigma-C_0$  for some constant $C_0$ near boundary, 
  which yields the completeness of the metric $e^u\omega$.
  
  \end{proof}

 \section{Equations on complete noncompact Hermitian manifolds}
  \label{sec1-noncompactcomplete}
  
In this section we solve the equation \eqref{main-equ2} on a complete noncompact Hermitian manifold.  
Together with Lemma \ref{lemma1-asymcondition}, we obtain Theorem \ref{theorem4-complete-noncompact}.

  \begin{theorem} 
  	\label{theorem1-complete-4} 
 	Let $(M,\omega)$ be a complete noncompact   Hermitian manifold. 
 	Let 
 	$\psi(z,t)$ be a smooth function on $ M\times \mathbb{R}$. 
  	Assume, in addition to
  	$\Gamma_\infty=\mathbb{R}^{n-1}$, $\sup_\Gamma f=+\infty$,  that $f$ satisfies  \eqref{addistruc}
  	in $\Gamma$.   Suppose $\psi(z,t)$ satisfies \eqref{nondegenerate-assumption-1},
  	\eqref{assump5-psi} and  $\psi_t(z,t)\geq0$. 
  	Assume in addition that there is an admissible function $\underline{u}\in C^2(M)$ such that
  	\begin{equation}
  		\label{asymptotic-assumption3}
  		\begin{aligned}
  			f(\lambda (\chi+\sqrt{-1}\partial\overline{\partial}
  			\underline{u})) \geq  \psi  (z,\underline{u})
  			\,  \mbox{ in } M\setminus K_0
  		\end{aligned}
  	\end{equation}
  	where $K_0$ is a compact subset of $M$. 
  	Then there is 
  	an admissible function $u\in C^\infty(M)$ satisfying \eqref{main-equ2}.
  	Moreover,  $u$ is the maximal solution and $u\geq \underline{u}$.
  	
  \end{theorem}

  \begin{remark} When $(M,\omega)$ is complete noncompact, in  assumption \eqref{assump5-psi}  $\bar M$ shall be replaced by $M.$ \end{remark}

  
  \begin{proof}

  
  	
  	Fix an exhausting sequence 
  	$\{ M_k\}_{k=1}^{+\infty}$  by   
  	complex submanifolds  of complex dimension $n$  with smooth boundary such that 
  		$M=\cup_{k=1}^{\infty} M_k, \mbox{  }  \bar M_k =M_k\cup\partial M_k,  
  		\mbox{  }  \bar M_k\subset\subset M_{k+1}$. 
  	For any integer $k\geq 1$ we denote $u^{(k)}$ the admissible  solution to 
  	\begin{equation}
  		\begin{aligned}
  			f(\lambda(\chi+\sqrt{-1}\partial\overline{\partial}u^{(k)}))
  		=	\psi(z,u^{(k)}) \mbox{ in } M_k, \,\, 
  			\lim_{z\to \partial M_k} u^{(k)}(z)=+\infty.  \nonumber
  		\end{aligned}
  	\end{equation}
  	Moreover, $u^{(k)}\in C^\infty(M_k)$. 
  	The existence and regularity 
  	 follow from Theorem \ref{theorem1-complete-2}.
  	
  	By the maximum principle, we deduce that 
  	$$u^{(k)}\geq u^{(k+1)}  \mbox{ in } M_k.$$
  	On the other hand,
  	using the maximum principle again,  
   $$u^{(k)}\geq  \underline{u}  	\mbox{ in } M_k.$$
  	Let's take $u=\underset{k\to+\infty} \lim u^{(k)}.$
  	Such a limit exists and $u\geq \underline{u} 
  	$.
  	 In addition, 
  	$u\in C^\infty(M)$ according to
  	Evans-Krylov theorem,  Schauder theory, and the
 interior estimates 
  	(Proposition \ref{thm0-inter}).
  	Moreover, by the maximum principle, $u$ is the maximal solution. 
  	
 \end{proof}

	%
  %

  	 
  

  

  \section{Quantitative boundary estimate}
  \label{Sec-Estimates1}
  
We establish quantitative boundary estimate \eqref{quantitative-BE}, assuming local admissible function $\underline{u}$ satisfying \eqref{admifunction1-local}  near boundary, instead of existence of subsolution. 

  Pick $p_0\in \partial M$ and let $M_\delta$ be as in \eqref{M-delta}.
  We choose local 
  coordinates 
  \begin{equation}
  	\label{goodcoordinate1}
  	\begin{aligned}
  		(z_1,\cdots, z_n), \mbox{  } z_i=x_i+\sqrt{-1}y_i 
  	\end{aligned}
  \end{equation}
  centered at $p_0$ in a neighborhood which we assume to be contained in $M_\delta$ 
  such that at $p_0$ $(z=0)$, $g_{i\bar j}(0)=\delta_{ij}$ and $\frac{\partial}{\partial x_{n}}$ is the interior normal to $\partial M$. 
  Denote
  \begin{equation}
  	\begin{aligned}
  		\Omega_{\delta}=\{z\in M: |z|<\delta\}.
  	\end{aligned}
  \end{equation}
  Throughout this section the Greek letters $\alpha, \beta$ run from $1$ to $n-1$. 
  
  
  The quantitative boundary estimate consists of the following two propositions.
  \begin{proposition}
  	\label{yuan-k2v} 
Assume 
$f$ satisfies  \eqref{unbounded-1} and	\eqref{addistruc}.  
  	Suppose   near boundary  that
  	 there is a local admissible function $\underline{u}$    satisfying \eqref{admifunction1-local}.
  	Then for any admissible solution $u\in C^2(\bar M)$ to \eqref{main-equ2}-\eqref{bdy-condition2}, there is a uniform positive constant $C$ 
such that 
  	\begin{equation}
  		\begin{aligned}
  			\mathfrak{g}_{n\bar n}(p_0)\leq C\left(1+\sum_{\alpha=1}^{n-1}|\mathfrak{g}_{\alpha\bar n}(p_0)|^2\right) \, \mbox{ for } p_0\in\partial M.  
  		\end{aligned}
  	\end{equation}
  \end{proposition}
  
  \begin{proposition}
  	\label{mix-general}
  	
  	Suppose  near boundary that there is a local admissible function  $\underline{u}$ obeying \eqref{admifunction1-local}. 
Assume  
  	\eqref{addistruc} and \eqref{unbounded-1} hold. 
  	Then 
  	for any admissible solution $u\in C^3(M)\cap C^2(\bar M)$ to  Dirichlet problem
  	 \eqref{main-equ2}-\eqref{bdy-condition2},
  	there is a uniform positive constant $C$ 
  	 depending on 
  $|u|_{C^0(\bar M)}$ 	
  		and other known data under control,  
  	such that
  	\begin{equation}
  		\label{quanti-mix-derivative-00} 
  		|\mathfrak{g}_{\alpha \bar n}(p_0)|\leq C(1+\sup_{M}|\partial u|)\, \mbox{ for } p_0\in\partial M. 
  	\end{equation}
  \end{proposition}
  

\subsection{Preliminaries}

Denote $F(\chi+\sqrt{-1}\partial\overline{\partial}u):= f(\lambda(\chi+\sqrt{-1}\partial\overline{\partial}u)).$
The linearized operator of   \eqref{main-equ2} at $u$, 
say $\mathcal{L}$,
 is  locally 
 given by
\begin{equation}
	\mathcal{L} v=F^{i\bar j} v_{i\bar j}  \nonumber
\end{equation}
where $F^{i\bar j}=\frac{\partial F(\mathfrak{g})}{\partial \mathfrak{g}_{i\bar j}},$
$\mathfrak{g}_{i\bar j}=\chi_{i\bar j}+u_{i\bar j}.$
Moreover, we denote $ {\lambda}=\lambda (\mathfrak{{g}})$ and 
\begin{align*} 
	\mathfrak{\underline{g}}_{i\bar j}=\,&
	\chi_{i\bar j} + \underline{u}_{i\bar j},  \,\,  
	\underline{\lambda}=\lambda (\mathfrak{\underline{g}}). 
\end{align*}  
We have standard identities $F^{i\bar j}\mathfrak{g}_{i\bar j}=\sum_{i=1}^n f_i(\lambda)\lambda_i, \,  F^{i\bar j}g_{i\bar j}=\sum_{i=1}^n f_i(\lambda).$ 

Let $\underline{u}$ be the local admissible function  given by \eqref{admifunction1-local}.
Since $f$ satisfies the unbounded condition \eqref{unbounded-1},  
we can check that $\underline{u}$ 
is a local $\mathcal{C}$-subsolution,
introduced by \cite{Guan12a} and \cite{Gabor},   of \eqref{main-equ2} near boundary.
This allows us to apply the following lemma, 
which is 
a refinement of
\cite[Theorem  2.18]{Guan12a}. 
We also refer to \cite{Guan-Dirichlet} for more analogue result.
\begin{lemma}
	[\cite{Gabor}]
	\label{guan2014}
	There exist positive constants $R_0$, $\varepsilon$ 
	such that if $|\lambda|\geq R_0$  then we either have
	\begin{equation}
		\label{gabor-case1}
		\begin{aligned}
			F^{i\bar j}(\underline{\mathfrak{g}}_{i\bar j}-\mathfrak{g}_{i\bar j})
			\geq \varepsilon F^{i\bar j}g_{i\bar j}  \nonumber
		\end{aligned}
	\end{equation}
	or
	\begin{equation}
		\label{gabor-case2}
		\begin{aligned}
			F^{i\bar j}\geq \varepsilon (F^{p\bar q}g_{p\bar q})g^{i\bar j}.
			\nonumber
		\end{aligned}
	\end{equation}
\end{lemma}


Let $\underline{u}$ be the local admissible function near boundary as  in \eqref{admifunction1-local}.
Then
\begin{equation}
	\label{yuan3-buchong5}
	\begin{aligned}
		u_{\alpha\bar\beta}(0)=\underline{u}_{\alpha\bar\beta}(0)+(u-\underline{u})_{x_n}(0)\sigma_{\alpha\bar\beta}(0).
	\end{aligned}
\end{equation}
Also this gives the bound of
second estimates for pure tangential derivatives
\begin{equation}
	\label{ineq2-bdy}
	\begin{aligned}
		|u_{\alpha\bar\beta}(0)|\leq C.
	\end{aligned}
\end{equation}

  \subsection{Double normal derivative case}

    We assume that $\Gamma$ is of type 1. Then $\Gamma_\infty$ is a symmetric convex cone as noted in \cite{CNS3}. 
  (For the type 2 case, see Proposition \ref{RK6.3}).
  
  At $p_0$ ($z=0$),  by \eqref{yuan3-buchong5} we have 
  \begin{equation}
  	\label{410-buchong}
  	\begin{aligned}
  		\mathfrak{g}_{\alpha\bar\beta}= (1-t)\underline{\mathfrak{g}}_{\alpha\bar\beta}
  		+\{t\underline{\mathfrak{g}}_{\alpha\bar\beta}+ (u-\underline{u})_{x_n}\sigma_{\alpha\bar\beta}\}.
  	\end{aligned}
  \end{equation}
  For simplicity, 
  we denote
  \begin{equation}
  	\label{A_t}
  	\begin{aligned}
  		A_t=\sqrt{-1} \left[t\underline{\mathfrak{g}}_{\alpha\bar\beta}+ (u-\underline{u})_{x_n}\sigma_{\alpha\bar\beta}\right]dz_\alpha\wedge d\bar z_\beta.
  	\end{aligned}
  \end{equation}
  Clearly, from \eqref{410-buchong}
  $(A_{1})_{\alpha\bar\beta}=\mathfrak{g}_{\alpha\bar\beta}$. 
    Let $t_0$ be the first  
  $t$ as we decrease $t$ from $1$  
  such that
  \begin{equation}
  	\label{key0-yuan3}
  	\begin{aligned}
  		\lambda_{\omega'}(A_{t_0})\in\partial\Gamma_\infty.
  	\end{aligned}
  \end{equation}
Henceforth, $\lambda_{\omega'}(\chi')$ denotes the eigenvalues of $\chi'$ with respect to $\omega'=\sqrt{-1}g_{\alpha\bar\beta} dz_\alpha\wedge d\bar z_\beta.$
  Such $t_0$ exists, since 
  $\lambda_{\omega'}(A_1)\in\Gamma_\infty$ and
  $\lambda_{\omega'}(A_t)\in \mathbb{R}^{n-1}\setminus\Gamma_\infty$ for $t\ll -1$. Furthermore, for a uniform positive constant $T_0$ under control,
  \begin{equation}
  	\label{1-yuan3}
  	\begin{aligned}
  		-T_0< t_0<1.
  	\end{aligned}
  \end{equation}
  
  Let   \begin{equation}
  	\label{underlambda-2}
  	\begin{aligned}
  		\underline{\lambda}'=\lambda_{\omega'}( \underline{\mathfrak{g}}_{\alpha\bar\beta}).
  	\end{aligned}
  \end{equation} 
Since $\underline{u}$ is admissible, there is $\varepsilon_0>0$ small such that
  \begin{equation} 
  	\label{varepsilon-pro0}
	\begin{aligned}
	 \underline{\lambda}-\varepsilon_0\vec{\bf 1}\in\Gamma.
	\end{aligned}
\end{equation}
  By the unbounded condition \eqref{unbounded-1}
  there is a uniform positive constant $R_1$ depending on $(1-t_0)^{-1}$,
  $\sup_{\partial M}\psi[u]$, 
  $\varepsilon_0$ and $\underline{\lambda}'$
  such that  	 
  \begin{equation}
  	\label{key-03-yuan3}
  	\begin{aligned} 	f\left((1-t_0)(\underline{\lambda}'_{1}-{\varepsilon_0}/{2}),\cdots, (1-t_0)(\underline{\lambda}'_{n-1}
  		-{\varepsilon_0}/{2}), R_1\right)\geq \psi[u], 
  	\end{aligned}
  \end{equation}
and $(\underline{\lambda}'_1-\varepsilon_{0}, \cdots, \underline{\lambda}'_{n-1}-\varepsilon_0, {(1-t_0)^{-1}}{R_1})\in \Gamma.$

  \vspace{1mm}
  Following the idea from  \cite{CNS3}  (refined by \cite{LiSY2004} in complex variables), for such $t_0$
one can  prove  that 
  \begin{lemma}
  	\label{keylemma1-yuan3}
  	There is a uniform positive constant  $C$ depending on 
  	$|u|_{C^0(M)}$, $|\partial u|_{C^0(\partial M)}$,  $\inf_M\psi[u]$,
  	$\partial M$ up to third derivatives and other known data, such that
  	\begin{equation}
  		\begin{aligned}
  			(1-t_0)^{-1}\leq C.  \nonumber
  		\end{aligned}
  	\end{equation}

  \end{lemma}


  Below we complete the proof of Proposition \ref{yuan-k2v}. And we leave the proof of Lemma \ref{keylemma1-yuan3} to the end of this subsection.
  
  \subsubsection{Proof of Proposition \ref{yuan-k2v}}
  	
   The proof is based on Lemmas \ref{yuan's-quantitative-lemma} and \ref{lemma3.4}. 
  	Let $$ {A}(R) =
  	\left( \begin{matrix}
  		\mathfrak{{g}}_{\alpha\bar \beta} &\mathfrak{g}_{\alpha\bar n}\\
  		\mathfrak{g}_{n\bar \beta}& R  \nonumber
  	\end{matrix}\right).$$  
  	By 
  	\eqref{410-buchong} we can decompose $A(R)$ into
  	\begin{equation}
  		\begin{aligned}
  			{A}(R) = A'(R)+A''(R) 
  		\end{aligned}
  	\end{equation}
  	where 
  	\[A'(R)=\left(\begin{matrix}
  		(1-t_0)(\mathfrak{\underline{g}}_{\alpha\bar \beta}-\frac{\varepsilon_0}{4}\delta_{\alpha\beta}) &\mathfrak{g}_{\alpha\bar n}\\
  		\mathfrak{g}_{n\bar \beta}& R/2  \nonumber
  	\end{matrix}\right), \mbox{  }
  	A''(R)=\left(\begin{matrix} (A_{t_0})_{\alpha\bar \beta}+\frac{(1-t_0)\varepsilon_0}{4} \delta_{\alpha\beta} &0\\ 0& R/2  \end{matrix}\right).\]

We denote
  	\begin{equation}
  		\label{yuan3-buchong3}
  		\lambda_{\omega'}(A_{t_0}):=
  		\tilde{\lambda}'=(\tilde{\lambda}_1',\cdots,\tilde{\lambda}_{n-1}').
  	\end{equation}
  	One can see that there is a uniform constant $C_0>0$ depending on $|t_0|$,
  	$\sup_{\partial M}|\partial u|$ and other known data, such that $|\tilde{\lambda}'|\leq C_0$, that is
  	$\tilde{\lambda'}$ is contained in a compact subset of $\overline{\Gamma}_\infty$, i.e.,
  	\begin{equation}
  		\label{yuan3-buchong4v}
  		\begin{aligned}
  			\tilde{\lambda}'\in {K}:=\{\lambda'\in \overline{\Gamma}_\infty: |\lambda'|\leq C_0\}.  \nonumber
  		\end{aligned}
  	\end{equation}
  	Thus  there is a 
  	uniform positive constant $R_2$ depending on $((1-t_0)\varepsilon_0)^{-1}$, $K$ and other known data, 
  	such that  
  	\begin{equation}
  		\label{at0-1}
  		\begin{aligned}
  			\lambda(A''(R))\in\Gamma, \mbox{ } \forall R>R_2.
  		\end{aligned}
  	\end{equation}

  	Let's pick  $\epsilon=\frac{(1-t_0)\varepsilon_0}{4}$ in 
  	Lemma  \ref{yuan's-quantitative-lemma}, then 
  	as in \cite{yuan-regular-DP} we set
  	\begin{equation}
  		\begin{aligned}
  		 R_c= \,&
  		  \frac{8(2n-3)}{(1-t_0)\varepsilon_0}\sum_{\alpha=1}^{n-1} | \mathfrak{g}_{\alpha\bar n}|^2 	
  		 + 2(n-1) (1-t_0) \sum_{\alpha=1}^{n-1}|{\underline{\lambda}_\alpha'}|
	 \\\,&	
  		 		 +\frac{n(n-1)(1-t_0)\varepsilon_0}{2}
  		 	+2R_1+2R_2  \nonumber
  		\end{aligned}
  	\end{equation}
  	where $\varepsilon_0$, $R_1$ and $R_2$ are the constants as we 
  	fixed in \eqref{key-03-yuan3} and \eqref{at0-1}.
  	
  	According to Lemma \ref{yuan's-quantitative-lemma},
  	the eigenvalues $\lambda({A}'(R_c))$ of ${A}'(R_c)$
  	shall behave like
  	\begin{equation}
  		\label{lemma12-yuan}
  		\begin{aligned}
  			\,& \lambda_{\alpha}({A}'(R_c))\geq (1-t_0)(\underline{\lambda}'_{\alpha}-\frac{\varepsilon_0}{2}), \mbox{  } 1\leq \alpha\leq n-1, \\\,&
  			\lambda_{n}({A}'(R_c))\geq R_c/2-(n-1)(1-t_0)\varepsilon_0/4.
  		\end{aligned}
  	\end{equation}
  	In particular, $\lambda({A}'(R_c))\in \Gamma$. So $\lambda(A(R_c))\in \Gamma$. 
  	%
  	Together with 
  	\eqref{concavity2},
  	Lemma \ref{lemma3.4} yields that
  	\begin{equation}
  		\label{yuan-k1}
  		\begin{aligned}
  			f(\lambda(A(R_c)))\geq f(\lambda(A'(R_c))). 
  		\end{aligned}
  	\end{equation}
  	From \eqref{key-03-yuan3}, \eqref{lemma12-yuan} and \eqref{yuan-k1}, we deduce
  	  $\mathfrak{g}_{n\bar n} \leq R_c.$  
  	

\begin{proposition}
	\label{RK6.3}
	 When $f$ is of uniform ellipticity, 
	 we have a more delicate 
	 estimate
	 \begin{equation}
	 	\label{QBE-2}
	 	\begin{aligned}
	 		\mathfrak{g}_{n\bar n}(p)\leq C\left(1+\sum_{\alpha=1}^{n-1}|\mathfrak{g}_{\alpha\bar n}(p)|\right), \,\, \forall p\in\partial M.  
	 	\end{aligned}
	 \end{equation}
 \end{proposition}
 \begin{proof} 
 	 Fix $p_0\in\partial M$.  
 	We assume 
 	$\mathfrak{g}_{nn}(p_0)\geq  1$.
From \eqref{uniform-elliptic-2}, 
 	$F^{n\bar n}\geq  \theta  F^{i\bar j}g_{i\bar j}. $
 	Let $C_{\sup\psi[u]}$ be the positive constant with
 	\begin{equation}
 		\label{concve-inequ2}
 		f(C_{\sup\psi[u]}\vec{\bf 1})=\sup_{z\in M}\psi[u](z).  
 	\end{equation} 
The concavity yields 
$	0\geq   		
 				F^{i\bar j} (\mathfrak{g}_{i\bar j}-C_{\sup\psi[u]}\delta_{ij}).  $
Together with \eqref{ineq2-bdy}, we get \eqref{QBE-2}.
 	\end{proof}


 \subsubsection{Proof of  Lemma \ref{keylemma1-yuan3}} \label{proof-keylemma1}

	
We follow closely  \cite{LiSY2004}.
We assume that $\Gamma$ is of type 1. Then $\Gamma_\infty$ is a symmetric convex cone as noted by \cite{CNS3}. 
(The case of type 2 cone is much more simpler since $\Gamma_\infty=\mathbb{R}^{n-1}$). 
The  proof presented below is a slight modification of that in \cite{yuan-regular-DP}.

Let $\check{u}$ and $\underline{u}$ be as in \eqref{supersolution1} and \eqref{admifunction1-local}, respectively.
Let  
$\underline{\lambda}'$ and $\tilde{\lambda}'$   be as in \eqref{underlambda-2} and   \eqref{yuan3-buchong3}, respectively. 
For simplicity, we denote $$\eta=(u-\underline{u})_{x_n}(0).$$ 
We assume $\eta>0$ (otherwise we are done). 
Without loss of generality, we assume $$t_0>\frac{1}{2} \mbox{ and } \tilde{\lambda}_1'\leq \cdots \leq \tilde{\lambda}_{n-1}'.$$ 
It was proved in \cite[Lemma 6.1]{CNS3} that for $\tilde{\lambda}'\in\partial\Gamma_\infty$
there is a supporting plane for
$\Gamma_\infty$ and one can choose $\mu_j$ with 
$\mu_1\geq \cdots\geq \mu_{n-1}\geq0$ so that
\begin{equation}
	\label{key-18-yuan3}
	\begin{aligned}
		\Gamma_\infty\subset \left\{\lambda'\in\mathbb{R}^{n-1}: \sum_{\alpha=1}^{n-1}\mu_\alpha\lambda'_\alpha>0 \right\}, \mbox{  }
		\mbox{  } \sum_{\alpha=1}^{n-1} \mu_\alpha=1, \mbox{  } \sum_{\alpha=1}^{n-1}\mu_\alpha \tilde{\lambda}_\alpha'=0.
	\end{aligned}
\end{equation} 

Note that as in \eqref{varepsilon-pro0},
$\underline{\lambda}-\varepsilon_0\vec{\bf1}\in\Gamma.$  
Then  $(\underline{\lambda}'_1-\varepsilon_0,\cdots,\underline{\lambda}'_{n-1}-\varepsilon_0)\in\Gamma_\infty$ and so
\begin{equation}
	\label{adf}
	\begin{aligned}
		\sum_{\alpha=1}^{n-1}\mu_\alpha \underline{\lambda}'_\alpha
		\geq \varepsilon_0>0. 
	\end{aligned}
\end{equation} 
According to 
\cite[Lemma 6.2]{CNS3}
(without loss of generality,  assume $\underline{\lambda}_1'\leq \cdots\leq\underline{\lambda}_{n-1}'$), 
\begin{equation}
	\begin{aligned}
		\sum_{\alpha=1}^{n-1} \mu_\alpha \underline{\mathfrak{g}}_{\alpha\bar\alpha}\geq \sum_{\alpha=1}^{n-1}\mu_\alpha \underline{\lambda}'_\alpha  
		\geq \varepsilon_0. 
	\end{aligned}
\end{equation} 
Without loss of generality,  we assume $({A_{t_0}})_{\alpha\bar\beta}=t_0\underline{\mathfrak{g}}_{\alpha\bar\beta}+\eta\sigma_{\alpha\bar\beta}$ is diagonal at $p_0$. From \eqref{key-18-yuan3} one has at the origin
\begin{equation}
	\begin{aligned}
		0=  t_0 \sum_{\alpha=1}^{n-1}\mu_\alpha\underline{\mathfrak{g}}_{\alpha\bar\alpha}+\eta\sum_{\alpha=1}^{n-1}\mu_\alpha  {\sigma}_{\alpha\bar\alpha} 
		>\frac{\varepsilon_0}{2} +\eta \sum_{\alpha=1}^{n-1} \mu_\alpha \sigma_{\alpha\bar\alpha}.
	\end{aligned}
\end{equation}
Together with \eqref{c0-boundary-c1} and $\eta=(u-\underline{u})_{x_n}(0)$, we see at the origin $(z=0)$
\begin{equation}
	\label{key-1-yuan3}
	\begin{aligned}
		-\sum_{\alpha=1}^{n-1} \mu_\alpha \sigma_{\alpha\bar\alpha}\geq
		\frac{\varepsilon_0}{2\sup_{\partial M}|\nabla (\check{u}-\underline{u})|}=:a_1>0.
	\end{aligned}
\end{equation}

On $\Omega_\delta=M\cap B_{\delta}(0)$, 
we take
\begin{equation}
	\begin{aligned}
		d(z)=\sigma(z)+\tau |z|^2
	\end{aligned}
\end{equation}
where $\tau$ is a positive constant 
to be determined. Let 
\begin{equation}
	\label{w-buchong1}
	\begin{aligned}
		w(z)=\underline{u}(z)+({\eta}/{t_0})\sigma(z)+l(z)\sigma(z)+Ad(z)^2,
	\end{aligned}
\end{equation}
where $l(z)=\sum_{i=1}^n (l_iz_i+\bar l_{i}\bar z_{i})$, 
$l_i\in \mathbb{C}$, $\bar l_i=l_{\bar i}$,
to be chosen as in \eqref{chosen-1} below, and $A$ is a positive constant to be determined. 
Furthermore, 
\begin{equation} \label{u-w1}   \begin{aligned}
		u(z)-w(z)=-A\tau^2|z|^4 \mbox{ on } \partial M\cap\bar\Omega_\delta.
\end{aligned}    \end{equation}
When $A\gg1$, on $M\cap\partial B_{\delta}(0)$ we see  
\begin{equation}
	\label{u-w2}
	\begin{aligned}
		u(z)-w(z) 
		\leq 
		-(2A\tau \delta^2+\frac{\eta}{t_0}-2n \sup_{i}|l_i| \delta)\sigma(z)-A\tau^2\delta^4 
		\leq 
		-\frac{A\tau^2 \delta^4}{2}. 
	\end{aligned}
\end{equation}

Let $T_1(z),\cdots, T_{n-1}(z)$ be an orthonormal basis for holomorphic tangent space 
of level hypersurface
$ \{w: d(w)=d(z)\}$ at $z$, so that  for each 
$1\leq\alpha\leq n-1$, $T_\alpha$ is of $ C^1$ class and  
$T_\alpha(0)= \frac{\partial }{\partial z_\alpha}$. 
Such a basis exists and
the holomorphic tangent space 
can be characterized as  $\left\{\xi=\xi^i\frac{\partial}{\partial z^i}: (\sigma_i+\tau \bar z_i ) \xi^i=0\right\}$,
see e.g. \cite{Demailly-CADG-book}. 
%
%

By   \cite[Lemma 6.2]{CNS3}, we have the following lemma.
\begin{lemma}
	\label{lemma-yuan3-buchong1}
	Let $T_1(z),\cdots, T_{n-1}(z)$ be as above, and  let  $T_n=\frac{\partial d}{|\partial d|}$.
	For a real $(1,1)$-form $\Theta=\sqrt{-1}\Theta_{i\bar j}dz_i\wedge d\bar z_j$,
	we denote by $\lambda(\omega^{-1}\Theta)=(\lambda_1(\Theta),\cdots,\lambda_n(\Theta))$  the eigenvalues of $\Theta$ 
	with $\lambda_1(\Theta)\leq \cdots\leq \lambda_n(\Theta)$. Then for any 
	$\mu_1\geq\cdots\geq\mu_n$,
	\[\sum_{i=1}^n \mu_i \lambda_{i}(\Theta)\leq \sum_{i=1}^n\mu_i\Theta(T_i,J\bar T_i).\]
\end{lemma}

Let $\mu_1,\cdots,\mu_{n-1}$ be as in \eqref{key-18-yuan3}, and we set $\mu_n=0$. Let's denote $T_\alpha=\sum_{k=1}^nT_\alpha^k\frac{\partial }{\partial z_k}$. 
For $\Theta=\sqrt{-1}\Theta_{i\bar j}dz_i\wedge d\bar z_j$, we define
\begin{equation}
	\begin{aligned}
		\Lambda_\mu(\Theta):= \sum_{\alpha=1}^{n-1}\mu_{\alpha} T_{\alpha}^i \bar T_{\alpha}^j \Theta_{i\bar j}.  \nonumber
	\end{aligned}
\end{equation}

\begin{lemma}
	\label{lemma-key2-yuan3}
	Let $w$ be as in \eqref{w-buchong1}.
	There are parameters $\tau$, $A$, $l_i$, $\delta$ depending only on $|u|_{C^0(M)}$, 
	$|\nabla u|_{C^0(\partial M)}$,  
	$|\underline{u}|_{C^2(M)}$, 
	$\partial M$ up to third derivatives and other known data, such that 
	\begin{equation}
		\begin{aligned}
			\Lambda_\mu (\mathfrak{g}[w])   \leq0  \mbox{ in } \Omega_\delta, \,\, u\leq w \mbox{ on } \partial \Omega_\delta. \nonumber
		\end{aligned}
	\end{equation}
\end{lemma}

\begin{proof}
	By direct computation
	\begin{equation}
		\begin{aligned}
			\Lambda_\mu (\mathfrak{g}[w])
			=\,&\sum_{\alpha=1}^{n-1} \mu_\alpha T_{\alpha}^i \bar T_{\alpha}^j
			(\chi_{i\bar j}+\underline{u}_{i\bar j}+\frac{\eta}{t_0}\sigma_{i\bar j}) 
			+ 2Ad(z)\sum_{\alpha=1}^{n-1} \mu_\alpha T_{\alpha}^i \bar T_{\alpha}^j d_{i\bar j}
			\\\,&
			+ \sum_{\alpha=1}^{n-1} \mu_\alpha T_{\alpha}^i \bar T_{\alpha}^j (l(z)\sigma_{i\bar j}
			+l_i\sigma_{\bar j}+\sigma_i l_{\bar j}).  \nonumber
		\end{aligned}
	\end{equation}
	Here we use $T_\alpha d=0$  for $1\leq\alpha\leq n-1.$ 
	Next, we will estimate 
	$\Lambda_\mu (\mathfrak{g}[w])$ in $\Omega_\delta$.
	\begin{itemize}
		\item
		At the origin $(z=0)$, $T_{\alpha}^i=\delta_{\alpha i}$, 
		\begin{equation}
			\begin{aligned}
				\sum_{\alpha=1}^{n-1} \mu_\alpha T_{\alpha}^i \bar T_{\alpha}^j
				(\chi_{i\bar j}+\underline{u}_{i\bar j}+\frac{\eta}{t_0}\sigma_{i\bar j}) (0)
				=\frac{1}{t_0}\sum_{\alpha=1}^{n-1}\mu_\alpha (A_{t_0})_{\alpha\bar\alpha}=0.  \nonumber
			\end{aligned}
		\end{equation}
		So there are complex constants $k_i$ 
		such that
		\begin{equation}
			\begin{aligned}
				\sum_{\alpha=1}^{n-1} \mu_\alpha T_{\alpha}^i \bar T_{\alpha}^j
				(\chi_{i\bar j}+\underline{u}_{i\bar j}+\frac{\eta}{t_0}\sigma_{i\bar j}) (z)=\sum_{i=1}^n (k_i z_i+ \bar k_{i} \bar z_{i})+O(|z|^2). \nonumber
			\end{aligned}
		\end{equation}

		\item   
		Note that	\begin{equation}
			\label{yuan3-buchong2}
			\begin{aligned}
				\sum_{\alpha=1}^{n-1} \mu_\alpha T_{\alpha}^i \bar T_{\alpha}^j(z)
				=\sum_{\alpha=1}^{n-1} \mu_\alpha T_{\alpha}^i \bar T_{\alpha}^j(0)+O(|z|) =\sum_{\alpha=1}^{n-1} \mu_\alpha \delta_{\alpha i} \delta_{\alpha j}+O(|z|),
			\end{aligned}
		\end{equation}
		\begin{equation}
			\label{c3-yuan3}
			\begin{aligned}
				\sum_{\alpha=1}^{n-1}\mu_\alpha\sigma_{\alpha\bar\alpha}(z)=\sum_{\alpha=1}^{n-1}\mu_\alpha\sigma_{\alpha\bar\alpha}(0)+O(|z|).
			\end{aligned}
		\end{equation}  
		Combining with \eqref{key-1-yuan3}, one can pick $\delta, \tau$ sufficiently small such that 
		\begin{equation}
			\begin{aligned}
				\sum_{\alpha=1}^{n-1}\mu_\alpha T_{\alpha}^i \bar T_{\alpha}^j d_{i\bar j}
				=\,&   \sum_{\alpha=1}^{n-1} \mu_\alpha(\sigma_{\alpha\bar\alpha}(z)+\tau)
				+\sum_{\alpha=1}^{n-1}\mu_\alpha \left(T_{\alpha}^i \bar T_{\alpha}^j (z)- T_{\alpha}^i \bar T_{\alpha}^j (0)\right)d_{i\bar j} \\
				\leq\,& -a_1+\tau+O(|z|) \leq -\frac{a_1}{4}.  \nonumber
			\end{aligned}
		\end{equation}
		Consequently,
		\begin{equation}
			\begin{aligned}
				2Ad(z)\sum_{\alpha=1}^{n-1} \mu_\alpha T_{\alpha}^i \bar T_{\alpha}^j d_{i\bar j} \leq -\frac{a_1A}{2}d(z). \nonumber
			\end{aligned}
		\end{equation}
		
		\item  From 
		$ \sum_{i=1}^n T_\alpha^i \sigma_i=-\tau\sum_{i=1}^n T_\alpha^i
		\bar z_i $
		we have
		\begin{equation}
			\begin{aligned}
				\sum_{\alpha=1}^{n-1}\mu_\alpha T^i_\alpha \bar T^j_\alpha(l_i \sigma_{\bar j} + \sigma_i l_{\bar j})
				=-\tau \sum_{\alpha=1}^{n-1} \mu_\alpha(\bar l_{\alpha} \bar z_{\alpha}+l_\alpha z_\alpha)
				+O(|z|^2).   \nonumber
			\end{aligned}
		\end{equation}
		On the other hand, by \eqref{yuan3-buchong2}, 
		\begin{equation}	\begin{aligned}	l(z) \sum_{\alpha=1}^{n-1}\mu_\alpha T^i_\alpha \bar T^j_\alpha \sigma_{i\bar j} 	= l(z) \sum_{\alpha=1}^{n-1}\mu_\alpha\sigma_{\alpha\bar\alpha}(0)	+O(|z|^2).   \nonumber	\end{aligned}  	\end{equation}
		Thus 
		\begin{equation}
			\begin{aligned}
				\,& 
				l(z)\sum_{\alpha=1}^{n-1}\mu_\alpha T^i_\alpha \bar T^j_\alpha
				\sigma_{i\bar j}+\sum_{\alpha=1}^{n-1}\mu_{\alpha}T^i_\alpha \bar T^j_\alpha(l_i\sigma_{\bar j}
				+\sigma_i l_{\bar j}) \\
				= \,&
				l(z) \sum_{\alpha=1}^{n-1}\mu_\alpha\sigma_{\alpha\bar\alpha}(0)
				- \tau\sum_{\alpha=1}^{n-1} \mu_\alpha (z_\alpha l_\alpha+\bar z_{\alpha}  \bar l_{\alpha})
				+O(|z|^2). \nonumber
			\end{aligned}
		\end{equation}
		
	\end{itemize}
	
	Putting these together, 
	\begin{equation}
		\label{together1}
		\begin{aligned}
			\Lambda_\mu(\mathfrak{g}[w])\leq\,&
			\sum_{\alpha=1}^{n-1} 2\mathfrak{Re}
			\left\{z_\alpha \left(k_\alpha 
			-\tau\mu_\alpha l_\alpha
			+l_\alpha\sum_{\beta=1}^{n-1}\mu_\beta \sigma_{\beta\bar\beta}(0)\right)\right\}
			\\ \,& + 2\mathfrak{Re}\left\{z_n
			\left(k_n+l_n\sum_{\beta=1}^{n-1}\mu_\beta \sigma_{\beta\bar\beta}(0)\right)\right\} 
			-\frac{Aa_1}{2}d(z) + O(|z|^2).   \nonumber
		\end{aligned}
	\end{equation}
	Let $l_n=-\frac{k_n}{\sum_{\beta=1}^{n-1} \mu_\beta \sigma_{\beta\bar\beta}(0)}$.
	For $1\leq \alpha\leq n-1$, we set
	\begin{equation}
		\label{chosen-1}
		\begin{aligned}
			l_\alpha=-\frac{k_\alpha}{\sum_{\beta=1}^{n-1}\mu_\beta \sigma_{\beta\bar\beta}(0)-\tau \mu_\alpha}.
		\end{aligned}
	\end{equation}
	From $\mu_\alpha\geq 0$ and \eqref{key-1-yuan3}, we see
	such $l_i$ (or equivalently the  $l(z)$) are all well defined and uniformly bounded. 
	We thus complete the proof if $0<\tau, \delta\ll1$, $A\gg1$.
\end{proof}

\subsubsection*{\bf Completion of the proof of Lemma \ref{keylemma1-yuan3}}
Let $w$ be as in \eqref{w-buchong1}.
From the construction above, we know that there is a uniform positive constant $C_1'$ such that
$$ |\mathfrak{g}[w]|_{C^0(\Omega_\delta)}\leq C_1'.$$
Denote $\lambda[w] 
=\lambda(\mathfrak{g}[w])
$ and assume 
$\lambda_1[w]\leq \cdots\leq \lambda_n[w]$.
Together with Lemma \ref{lemma-yuan3-buchong1},   		Lemma \ref{lemma-key2-yuan3} implies
\begin{equation}
	\begin{aligned}
		\sum_{\alpha=1}^{n-1} \mu_\alpha \lambda_\alpha[w]\leq 0 \mbox{ in } \Omega_\delta.  \nonumber
	\end{aligned}
\end{equation}
So  by \eqref{key-18-yuan3} $(\lambda_1[w],\cdots,\lambda_{n-1}[w])\notin\Gamma_\infty$. In other words, $\lambda[w]\in X$, where
\[X:= \left( \mathbb{R}^{n}\setminus \Gamma \right)\cap \{\lambda\in\mathbb{R}^n:   |\lambda|\leq C_1'\}.\]
Consequently, $X\cap \bar\Gamma^{\inf_M\psi[u]}=\emptyset$, where 
$\bar\Gamma^{\inf_M\psi[u]}=\{\lambda\in\Gamma: f(\lambda)\geq \inf_M\psi[u]\}$
is the closure of  $\Gamma^{\inf_{M}\psi[u]}$.
Notice that   
$X$ is a compact subset. 
So we can deduce that the distance between $\bar\Gamma^{\inf_M\psi[u]}$ and $X$ 
has a uniform positive lower bound.
Therefore, there exists a positive constant $\epsilon_0$  such that 
\begin{equation}
	\label{distance1-positive}
	\begin{aligned}
		\epsilon_0\vec{\bf 1} +\lambda[w](z)\notin \bar\Gamma^{\inf_M\psi[u]}, \,\, \forall z\in\Omega_\delta.  
	\end{aligned}
\end{equation}

Near the origin $p_0$, 
under   coordinates
\eqref{goodcoordinate1} 
the distance can be expressed as 
\begin{equation}
	\begin{aligned}
		\sigma(z)=x_n+\sum_{ i,j=1}^{2n} a_{ij} t_it_j+O(|t|^3). \nonumber
	\end{aligned}
\end{equation}
Thus one can choose a  positive constant $C'$ such that $ x_n\leq C'|z|^2$ on 
$\partial M\cap \bar{\Omega}_\delta$. As a result,  
there is a positive constant $C_2$ depending only on 
$\partial M$ 
and $\delta$ so that  
$$x_n\leq C_2 |z|^2 \mbox{ on }\partial\Omega_\delta.$$

From \eqref{distance1-positive} we can take ${h}(z)=w(z)+\epsilon  (|z|^2-\frac{x_n}{C_2}) $ for some 
$ \epsilon>0,$ 
such that 
\begin{equation}
	\begin{aligned}
		\lambda[h](z)\notin \bar\Gamma^{\inf_M\psi[u]}, \,\, \forall z\in\Omega_\delta.
	\end{aligned}
\end{equation}
Moreover, from \eqref{u-w1}-\eqref{u-w2} we know $u\leq {h} \mbox{ on } \partial \Omega_\delta.$
The maximum principle (e.g.  \cite[Lemma B]{CNS3}) implies 
$$u\leq {h} \mbox{ in }\Omega_\delta.$$
Notice 
$u(0)=\varphi(0)$ and ${h}(0)=\varphi(0)$, we have $u_{x_n}(0)\leq h_{x_n}(0)$. Thus
\begin{equation}
	\begin{aligned}
		t_0\leq \frac{1}{1+\epsilon/(\eta C_2)}, \mbox{ i.e., }  (1-t_0)^{-1}\leq 1+\frac{\eta C_2}{\epsilon}.  \nonumber
	\end{aligned}
\end{equation}


  \medskip
  \subsection{Tangential-Normal derivatives case}




 In establishing Proposition \ref{mix-general}, we 
 use the
 subsolution method
of \cite{Guan1993Boundary,Guan1998The}  (further refined by  \cite{Guan12a}). 
In order to derive \eqref{quanti-mix-derivative-00}, we shall construct more delicate local barriers near boundary.  
 The  specific instance of such local barriers was investigated  by
\cite{Boucksom2012,Chen,Phong-Sturm2010} for   complex Monge-Ampère equation, and further by \cite{Collins2019Picard}
for more general  complex $k$-Hessian equations.  
The utilization of local barriers for general equations was introduced in \cite{yuan2020PAMQ,yuan2017,yuan-regular-DP}. 
 
  Fix $p_0\in\partial M$.
  Under   local coordinate \eqref{goodcoordinate1} centered at $p_0 $ $(z=0)$,  
  for convenience 
  \begin{equation}
  	t_{2k-1}=x_{k}, \ t_{2k}=y_{k},\ 1\leq k\leq n-1;\ t_{2n-1}=y_{n},\ t_{2n}=x_{n}.  \nonumber
  \end{equation}
We use notation as in  \eqref{formula-1}.
   By direct calculations, one derives  
  \begin{equation}
  	\begin{aligned}
  	  (u_{x_k})_{\bar j} 
  		=u_{\bar j x_k}+\overline{\Gamma_{kj}^l} u_{\bar l},  \,\,
  	  (u_{y_k})_{\bar j}
  		=u_{\bar j y_k}-{\sqrt{-1}}\overline{\Gamma_{kj}^l} u_{\bar l}, \nonumber
  	\end{aligned}
  \end{equation}
  \begin{equation}
  	\begin{aligned}
  	 (u_{x_k})_{i\bar j}  
  		=	 	u_{i\bar j x_k}+\Gamma_{ki}^lu_{l\bar j}+\overline{\Gamma_{kj}^l} u_{i\bar l}, 
  		\,\, \nonumber
  		(u_{y_k})_{i\bar j} 
  		=u_{i\bar j y_k}
  		+\sqrt{-1}		\left(\Gamma_{ki}^l u_{l\bar j}-\overline{\Gamma_{kj}^l} u_{i\bar l}\right).	\nonumber  	\end{aligned} \end{equation}

Let
  $\underline{u}$ be the local admissible function
  satisfying  \eqref{admifunction1-local}. 
  We set
  \begin{equation}
  	\label{barrier1}
  	\begin{aligned}
  		w= (\underline{u}-u)
  		- t\sigma
  		+N\sigma^{2}   \mbox{  in  } \Omega_{\delta}. \nonumber
  	\end{aligned}
  \end{equation}
 Here $N$ is a positive constant to be determined, $\delta$ and $t$ are small enough such that   
 \begin{equation}
 	\label{bdy1}
 	\begin{aligned}
 		\mbox{  $\sigma$ is smooth, }
 		\frac{1}{4} \leq |\partial \sigma|\leq 2,  \,\,
 		|\mathcal{L}\sigma | \leq   C_\sigma\sum_{i=1}^n f_i 
 		\mbox{ in }  \Omega_{\delta},
 	\end{aligned}
 \end{equation} 
 \begin{equation}
 	\label{yuanbd-11}
 	\begin{aligned}
 		N\delta-t\leq 0, \mbox{ }	 \max\{|2N\delta-t|, t\}
 		\leq \frac{\varepsilon}{16 C_\sigma},
 	\end{aligned}
 \end{equation}
for some  constant $C_\sigma>0,$
  where 
  $\varepsilon$ is the constant asserted in Lemma \ref{guan2014}. 
Clearly,
  \begin{equation}
  	\label{yiqi-1}
  	\begin{aligned}
  	  w\leq 0    \mbox{ in } \Omega_{\delta}, 
   \end{aligned}  \end{equation} \begin{equation}	\begin{aligned}
		\mathcal{L}w
		= F^{i\bar j}(\mathfrak{\underline{g}}_{i\bar j}-\mathfrak{g}_{i\bar j}) 	+2NF^{i\bar j}\sigma_i\sigma_{\bar j}	+(2N\sigma-t)\mathcal{L}\sigma
		\mbox{ in } \Omega_{\delta}. 
	\end{aligned}
\end{equation}

 We define the tangential operator on the boundary
\begin{equation}
	\label{tangential-oper-general1}
	\begin{aligned} 
		\mathcal{T}=\nabla_{\frac{\partial}{\partial t_{\alpha}}}- \widetilde{\eta}\nabla_{\frac{\partial}{\partial x_{n}}}, 
		\mbox{ for each fixed }
		1\leq \alpha< 2n,  \nonumber
	\end{aligned}
\end{equation}
where $\widetilde{\eta}=\frac{\sigma_{t_{\alpha}}}{\sigma_{x_{n}}}$. 
One has $\mathcal{T}(u-\varphi)=0$ on $\partial M\cap \bar\Omega_\delta$. 
By 
$\widetilde{\eta}(0)=0$ one derives $|\widetilde{\eta}|\leq C'|z|$ on $\bar\Omega_\delta$. 
Since  $(u-\varphi)\big|_{\partial M}=0$  we obtain
$\mathcal{T}(u-\varphi)\big|_{\partial M}=0$. Together with the boundary gradient estimate contained in \eqref{c0-boundary-c1},
one has
\begin{equation}\begin{aligned}\label{bdr-t}
		|(u-\varphi)_{t_{\alpha}}|\leq C|z| \mbox{ on } \partial M\cap\bar \Omega_\delta,
		\mbox{  } \forall 1\leq \alpha<2n.
	\end{aligned}
\end{equation}

  Denote  $b_{1}=1+\sup_{M} |\partial u|^{2}.$
  Take   
  \begin{equation}
  	\label{Phi-def1}
  	\begin{aligned}
  		\Phi=\pm \mathcal{T}(u-\varphi)+\frac{1}{\sqrt{b_1}}(u_{y_{n}}-\varphi_{y_{n}})^2 \,\, \mbox{ in } \Omega_\delta. \nonumber
  	\end{aligned}
  \end{equation}
Combining Cauchy-Schwarz inequality, we can prove
  \begin{equation}
  	\label{yuan-1}
  	\begin{aligned}
  		\mathcal{L}\Phi \geq 
  	 - C_{\Phi}  \sum_{i=1}^n f_i|\lambda_i|
  	 	-C_{\Phi}  \sqrt{b_1}  \sum_{i=1}^n f_i 
  		-C_\Phi  \sqrt{b_1}   \,\, \mbox{ in } \Omega_{\delta}.  
  	\end{aligned}
  \end{equation} 
Here we use 
\begin{equation}	\begin{aligned}		2|\mathfrak{Re}(F^{i\bar j}(\widetilde{\eta})_i (u_{x_{n}})_{\bar j})|	\leq  \,&  \frac{1}{\sqrt{b_1}}F^{i\bar j}(u_{y_n})_{i} (u_{y_n})_{\bar j} +C\sum_{i=1}^nf_i|\lambda_i|+C\sqrt{b_1}\sum_{i=1}^n f_i.		\nonumber	\end{aligned}\end{equation}

 By straightforward computations and \cite[Proposition 2.19]{Guan12a}, we have 
  \begin{equation}
  	\label{inequ-buchong1}
  	\begin{aligned}
  		\mathcal{L} \left(\sum_{\tau<n}|(u-\varphi)_{\tau}|^2 \right)  
  		\geq  \,&
  		\frac{1}{4}\sum_{i\neq r} f_{i}\lambda_{i}^{2}
  		-C_1'\sqrt{b_1} \left( \sqrt{b_1}+\sum_{i=1}^n f_{i} +\sum_{i=1}^n f_{i}|\lambda_{i}| \right). 
  		 \nonumber
  	\end{aligned}
  \end{equation}
From Lemma \ref{lemma3.4}  we see $\sum_{i=1}^n f_i\lambda_i\geq0$.
  Together with \eqref{concavity2}, we can prove that
  \begin{equation}
  	\label{concve-inequ1}
  	0\leq \sum_{i=1}^n f_i\lambda_i \leq C_{\sup\psi[u]}\sum_{i=1}^n f_i,
  \end{equation} 
  where $C_{\sup\psi[u]}$ is as in \eqref{concve-inequ2}.
  Combining Cauchy-Schwarz inequality, we have
  \begin{equation}
  	\label{flambda-0}
  	\begin{aligned}
  		\sum_{i=1}^n f_i |\lambda_i|
  		\leq  \frac{\epsilon}{4\sqrt{b_1}}\sum_{i\neq r} f_i\lambda_i^2 +\left(C_{\sup_M\psi[u]}+\frac{4\sqrt{b_1}}{\epsilon}\right)\sum_{i=1}^n f_i,  \mbox{ for } \epsilon>0. 
  	\end{aligned}
  \end{equation}
  On  the other hand, there is a uniform positive constant $\kappa_\sigma$ depending on $\sigma$ such that
  \begin{equation}
  	\label{sumfi1}	
  	\begin{aligned}	
  		\sum_{i=1}^n f_i(\lambda) \geq\kappa_\sigma, \mbox{ for }
  		f(\lambda)=\sigma.
  \end{aligned}\end{equation}

 One may construct in $\Omega_\delta$ the barrier function as follows:
  \begin{equation}
  	\label{Psi}
  	\begin{aligned} 
  		\widetilde{\Psi} =A_1 \sqrt{b_1}w -A_2 \sqrt{b_1} |z|^2 + A_3 \Phi+ \frac{1}{\sqrt{b_1}} \sum_{\tau<n}|(u-\varphi)_{\tau}|^2.  \nonumber
  	\end{aligned}
  \end{equation}
  Putting the above inequalities together, we obtain
  \begin{equation}
  	\label{bdy-main-inequality}
  	\begin{aligned}
  		\mathcal{L}\widetilde{\Psi} \geq \,&
  		A_1 \sqrt{b_1} F^{i\bar j}(\mathfrak{\underline{g}}_{i\bar j}-\mathfrak{g}_{i\bar j}) 
  		+ 2A_1N \sqrt{b_1} F^{i\bar j}\sigma_i \sigma_{\bar j}
  			-(C_1'+A_3 C_\Phi)   \sqrt{b_1} 
  		\\ \,&
  		-  \left( A_2+A_3C_\Phi  +A_1C_\sigma |2N\sigma-t|
  		+4(C_1'+A_3C_\Phi)^2 \right) \sqrt{b_1}\sum_{i=1}^n f_i
  		\\\,&
  		-\left(C_1'+(C_1'+A_3C_\Phi) C_{\sup\psi[u]} \right) \sum_{i=1}^n f_i.
  	\end{aligned}
  \end{equation}
  
  Proposition \ref{mix-general} follows from
   the following lemma.
  \begin{lemma}
  	There are constants $A_1\gg A_2\gg A_3\gg1$, $N\gg1$, $0<\delta\ll1$ such that  $\widetilde{\Psi}(0)=0$,
  	$\widetilde{\Psi}\big|_{\partial {\Omega_\delta}}\leq 0$, and 
\begin{equation}
	\label{mainiequality-1} \mathcal{L}\widetilde{\Psi}\geq 0 \mbox{ on } \Omega_\delta.
  	\end{equation}
  	
  \end{lemma}
  
  \begin{proof}
  Obviously,   $\widetilde{\Psi}(0)=0$.
  From \eqref{bdr-t} and \eqref{yiqi-1}, we see  
  	 $\widetilde{\Psi}\big|_{\partial {\Omega_\delta}}\leq 0$
  	if $A_2\gg A_3\gg1.$ 
  	Let $\varepsilon$ and $R_0$ be the corresponding positive constants in Lemma \ref{guan2014}. 
  	According to Lemma \ref{guan2014} the discussion can be divided into three cases.
  	
  	{\bf Case 1}: Assume that $|\lambda|\geq R_0$ and
  	\begin{equation}
  		\label{guan-key1}
  		\begin{aligned}
  			F^{i\bar j}(\mathfrak{\underline{g}}_{i\bar j}-\mathfrak{g}_{i\bar j})  \geq \varepsilon  \sum_{i=1}^n F^{i\bar i}.  \nonumber
  		\end{aligned}
  	\end{equation}
  	Note \eqref{yuanbd-11} implies 
  	$C_\sigma |2N\sigma-t|\leq \frac{1}{2} \varepsilon$.
  	 Taking $A_1\gg 1$   by \eqref{sumfi1}  we get \eqref{mainiequality-1}.
  	
  	{\bf Case 2}:   Suppose that  $|\lambda|\geq R_0$ and
   \begin{equation} 	\label{2nd-case1} 	\begin{aligned} 	f_{i} \geq  \varepsilon   \sum_{j=1}^n f_{j}, \,\, \forall 1\leq i\leq n.    	\end{aligned} 	\end{equation}
  	By \eqref{bdy1}, we have $|\partial \sigma|\geq\frac{1}{4}$ in $\Omega_\delta$, then 
  	\begin{equation}
  		\label{bbvvv}
  		\begin{aligned}
  			A_1N \sqrt{b_1} F^{i\bar j}\sigma_i \sigma_{\bar j} \geq \frac{A_1N \varepsilon \sqrt{b_1}}{16}\sum_{i=1}^n f_i  \mbox{ on } \Omega_\delta.
  		\end{aligned}
  	\end{equation}
  	This term  controls 
  	all the bad terms containing $\sum_{i=1}^n f_i$ in \eqref{bdy-main-inequality}.
  	On the other hand,  
  	$$\mathcal{L}(\underline{u}-u)\geq f(\underline{\lambda})-\psi(z,u)$$ 
  	and
  	the bad term  $-(C_1' +A_3 C_\Phi)\sqrt{b_1}$ from  \eqref{bdy-main-inequality}
  	can be dominated by combining \eqref{bbvvv} and \eqref{sumfi1}.
  	Thus  
  	\eqref{mainiequality-1} holds if $N\gg1$.
  	
  	{\bf Case 3}:  Assume $|\lambda|<R_0$. Then an inequality of the form \eqref{2nd-case1} holds with a possibly different constant $\varepsilon$. Consequently, this gives back  {\bf Case 2}.
  	
  \end{proof}

    \section{Interior estimates for  uniformly elliptic equations}
    \label{Sec-Estimates2}
  In this section we derive interior estimates for equations of uniform ellipticity.


   
    \begin{proposition}
   	\label{thm0-inter}
   	Let $B_r$ be a geodesic ball in $(M,\omega)$. Suppose  
   	\eqref{addistruc} and \eqref{uniform-elliptic-2} hold.
   	Then for any admissible solution $u\in C^4(B_r)$ to   \eqref{main-equ2} in $B_r$, we have 
   	\begin{equation}
   		\begin{aligned}
   			\sup_{B_{r/2}} (|\partial u|^2+|\partial\overline{\partial} u|)\leq C \nonumber
   		\end{aligned}
   	\end{equation}
   	where $C$ is a uniform constant depending only on $r^{-1}$,  $|u|_{C^0(B_r)}$, $|\psi|_{C^2(B_r)}$ and geometric quantities 
   	on $B_r$.
   	Moreover, \eqref{addistruc} can be removed when $\psi_u(z,u)\geq0$.
   \end{proposition}

\begin{remark}  
	For the equation \eqref{equation-n-varrho}, such interior estimates were established in \cite{GQY2018} for $\varrho<1$ and further extended by \cite{GGQ2022} to  the case    $\varrho=1$ when $f=\sigma_k^{1/k}$, $k<n$.  
		Together with Proposition \ref{proposition-n-varrho},
		 we are able to obtain interior estimates for general equation \eqref{equation-n-varrho}
	under the assumption $\varrho<\varrho_\Gamma$. 
 	This  partially answers a question 
 	 left open by \cite{GQY2018}. 
Moreover, as a contrast, in general one could not expect that such  interior estimates hold for 
the limiting case $\varrho=\varrho_\Gamma$. 
	 
\end{remark}

 \subsection{Useful formula}

  Denote  
 $w = |\partial u|^2 \mbox{ and } Q=|\partial\overline{\partial}u|^2+|\partial \partial u|^2.$
  Under local coordinates $z=(z_1,\cdots,z_n)$ around  $z_0$, with $g_{i\bar j}(z_0)=\delta_{ij}$, 
   we have 
   \begin{equation}
   	\label{deco1}
   	\begin{aligned}
   	 \,&	u_{i\bar j k}  
   	 	-u_{k\bar j i}  =T^l_{ik}u_{l\bar j}, \,\, w_i = u_{\bar k} u_{ki} + u_{k} u_{i\bar k}, \\ \nonumber 
   		u_{1\bar 1 i\bar i}-u_{i\bar i 1\bar 1} \,& =  R_{i\bar i 1\bar p}u_{p\bar 1}-
   		R_{1\bar 1 i\bar p}u_{p\bar i}   +2\mathfrak{Re}\{\bar T^{j}_{1i}u_{i\bar j 1}\}+T^{p}_{i1}\bar T^{q}_{i1}u_{p\bar q},  \\ \nonumber
 	w_{i\bar j} 
  =  u_{ki} u_{\bar k\bar j} +  u_{k\bar j} \,& u_{i\bar k} 
  +  u_{\bar k} u_{i\bar j k} +  u_{k} u_{i\bar j \bar k} 
  + R_{i\bar j k\bar l} u_{\bar k} u_l
  - T^{l}_{ik} u_{l\bar j} u_{\bar k} - \overline{T^{l}_{jk}} u_{i\bar l} u_k. \nonumber
   \end{aligned} \end{equation}
  \begin{lemma}
  	\label{lemma-Lw}
  	We have
  	\[F^{i\bar j}w_i w_{\bar j} \leq 2 wQ\sum F^{i\bar i};\]
  	and there exists $C>0$ such that
  	\begin{equation}
  		\begin{aligned}
  			\mathcal{L}(w) \geq   \frac{3\theta Q}{4} \sum F^{ii}-Cw\sum F^{i\bar i}-C|\nabla_z \psi| \sqrt{w}+2\psi_u w. \nonumber
  		\end{aligned}
  	\end{equation}
  	
  \end{lemma}

  \subsection{Interior gradient estimate}
  Let's consider the quantity 
  \[m_{0} = \max_{\bar M} \eta  |\partial u|^2 e^{\phi},\]
where $\eta$ is   as in \cite{Guan2003Wang} a smooth 
  function with compact support in $B_{r}\subset M$
  satisfying
  \begin{equation}
  	\label{2-22}
  	0\leq \eta \leq 1, ~~\eta\big|_{B_{\frac{r}{2}}}\equiv 1,
  	~~|\partial \eta | \leq \frac{C \sqrt{\eta}}{r},  
  	~~|\partial\overline{\partial}\eta| \leq \frac{C}{r^2} 
  \end{equation}
  and we take $\phi=v^{-N},$ where $v=u-\inf_{B_r} u+2$ 
  and 
    $N\gg1$ so that 
   \begin{equation}\label{5-15}
   	N(N+1)v^{-N-2}-N^2 v^{-2N-2} \geq N^2 v^{-N-2}.
   \end{equation}

  Suppose that $m_{0}$ is attained at an interior point $z_0 \in B_r$. 
  We choose local coordinates $(z_{1}, \ldots, z_{n})$ such that 
  $g_{i\bar j} = \delta_{ij}$ 
  at $z_0$.
  As above we denote $w = |\partial u|^2$. Without loss of generality, we assume $w(z_0)\geq 1$.  
  From above,  
  $\log \eta+\log w +\phi$ 
  achieves a maximum at $z_{0}$ and thus, 
  \begin{equation}
  	\label{gqy-G32}
  	\frac{\eta_{i}}{\eta}+\frac{w_{i}}{w}+\phi_{i} = 0, \;\;
  	\frac{\eta_{\bar i}}{\eta}+\frac{w_{\bar i}}{w}+\phi_{\bar i} = 0, \;\; \forall 1\leq i\leq n,
  \end{equation}
  \begin{equation}
  	\label{gqy-G34}
  	\mathcal{L}(\log \eta+\log w+\phi)\leq 0.
  \end{equation}
  
  Combining \eqref{gqy-G32} and Cauchy-Schwarz inequality, we derive
  \begin{equation}
  	\label{211-eq}
  	\frac{1}{w^2} F^{i\bar j}w_i w_{\bar j} \leq \frac{1+\epsilon}{\epsilon \eta^2} F^{i\bar j}\eta_i \eta_{\bar j}
  	+(1+\epsilon) F^{i\bar j} \phi_i \phi_{\bar j}.
  \end{equation}
  As a result, combining  Lemma \ref{lemma-Lw} and let  $8\epsilon\leq \theta$, we derive at $z_0$
  \begin{equation}
  	\label{L-logw}
  	\begin{aligned}
  		\mathcal{L} \log w 
  		\geq 
  	\left(\frac{\theta Q}{2w} -C\right)\sum F^{i\bar i}	
  		-\frac{1}{\epsilon \eta^2} F^{i\bar j} \eta_i \eta_{\bar j}
  		-F^{i\bar j} \phi_i \phi_{\bar j}
  		-\frac{C|\nabla_z \psi|}{\sqrt{w}} +2\psi_u. 
  	\end{aligned}
  \end{equation}
  On the other hand
  \begin{equation}
  	\label{2-23}
  	\begin{aligned}
  		\frac{1+\epsilon}{\epsilon \eta^{2}} F^{i\bar j} \eta_i \eta_{\bar j} 
  		- \frac{1}{\eta} F^{i\bar j} \eta_{i\bar j}  
  		\leq  \frac{C}{\epsilon r^2 \eta} \sum F^{i\bar i},	 
  	\end{aligned}
  \end{equation}
  \begin{equation}
  	\label{5-12}
  	\begin{aligned}
  		F^{i\bar i} \phi_i \phi_{\bar j}=N^2 v^{-2N-2} F^{i\bar j} u_iu_{\bar j},	 
  	\end{aligned}
  \end{equation}
  \begin{equation}
  	\label{Lphi}
  	\begin{aligned}
  		\mathcal{L}\phi
  		= N(N+1)v^{-N-2}F^{i\bar j} u_i u_{\bar j} -Nv^{-N-1}F^{i\bar j}    (\mathfrak{g}_{i\bar j}-\chi_{i\bar j}). 
  	\end{aligned}
  \end{equation}


  Plugging \eqref{concve-inequ1}, \eqref{5-15}, \eqref{211-eq}-\eqref{Lphi}
  into  \eqref{gqy-G34}, we obtain
  \begin{equation}
  	\begin{aligned}
  		\theta w N^2v^{-N-2} \sum F^{i\bar i}+\frac{\theta Q}{2w}\sum F^{i\bar i} 
  		\leq CN v^{-N-1} \sum F^{i\bar i}+ \frac{C}{r^2 \eta }\sum F^{i\bar i}
  		+\frac{C}{\sqrt{w}}-2\psi_u. \nonumber
  	\end{aligned}
  \end{equation}
  We can use \eqref{sumfi1} to control the term $-2\psi_u$.  
    As a result, we derive interior gradient estimate.
Furthermore, note that
  \[\frac{\theta w N^2v^{-N-2}}{2} \sum F^{i\bar i}+\frac{\theta Q}{2w}\sum F^{i\bar i} \geq \theta N v^{-\frac{N}{2}-1} \sqrt{Q}\sum F^{i\bar i}\]
  and there exists $R_0>0$ such that for any $\lambda$ with $|\lambda|\geq R_0$
  \begin{equation}
  	\label{sum-2}
  	\begin{aligned}
  		|\lambda|\sum_{i=1}^n f_i(\lambda) \geq \frac{f(|\lambda|\vec{\bf 1})-f(\lambda)}{2}
  		>0.
  \end{aligned}\end{equation}
  So one can remove \eqref{addistruc} 
  when $\psi_u\geq0$.

  \subsection{Interior estimate for second derivatives}
  

  	As in \cite{GQY2018} 
  	we  consider 
  	 the  quantity  
  	\begin{equation}
  		\begin{aligned}
  			P:= \sup_{z \in M} \max_{\xi \in T^{1,0}_z M}
  			\; e^{2 \phi} \mathfrak{g}_{p\bar q} \xi_p \bar{\xi_q}
  			\sqrt{g^{k\bar l}  \mathfrak{g}_{i\bar l}  \mathfrak{g}_{k\bar j}  \xi_i \bar{\xi_j}}/|\xi|^3 \nonumber
  		\end{aligned}
  	\end{equation}
  	where $\phi$ is a function depending on $z$ and $|\partial u|$. 
  	This  is inspired by \cite{Tosatti2019Weinkove}.
  	Assume that it is achieved at an interior point $p_0 \in M$ for some
  	$\xi \in T^{1,0}_{p_0} M$. 
  	By \cite[Lemma 2.9]{Steets2011Tian} we may choose local coordinates   $z=(z_{1},\cdots,z_{n})$ around $p_0$, such that at $p_0$, $g_{i\bj} = \delta_{ij}$,   and 
  	$$T_{ij}^k = 2 \Gamma_{ij}^k, \mbox{ }
  	\mathfrak{g}_{i\bar j}=\delta_{ij}\lambda_{i} 
  	\mbox{ and so } F^{i\bar j}=\delta_{ij}f_i.$$
  	 As in \cite{Tosatti2019Weinkove}, 
  	the maximum $P$ is achieved for $\xi = \partial_1$ at $p_0$.
  Assume $\mathfrak{g}_{1\bar{1}} \geq 1$;  otherwise we are done.
  
  	In what follows the computations are given at $p_0$.
  	Similar to the computations in
  	 \cite{GQY2018} one has
  	\begin{equation}
  		\label{mp1}
  		\mathfrak{g}_{1\bar{1} i} + \mathfrak{g}_{1\bar{1}} \phi_i = 0, \,\,
  		\mathfrak{g}_{1\bar{1}\bar i} +  \mathfrak{g}_{1\bar{1}} \phi_{\bar i} = 0,
  	\end{equation}
  	\begin{equation}
  		\label{gblq-C90}
  		\begin{aligned}
  			0  \geq   \frac{F^{i\bar i}   \mathfrak{g}_{1\bar{1} i\bar i}}{\mathfrak{g}_{1\bar 1}}
  			+ F^{i\bar i} (\phi_{i\bar i} - \phi_i \phi_{\bar i})
  			+ \frac{1}{8  \mathfrak{g}_{1\bar{1}}^2} \sum_{k>1} F^{i\bar i}
  			\mathfrak{g}_{1 \bar k i}   \mathfrak{g}_{k\bar{1} \bar i}
  			- C\sum F^{i\bar i}.
  		\end{aligned}
  	\end{equation}
  	
  	Combining the standard formula \eqref{deco1}, 
  	we can derive 
  	\begin{equation}
  		\begin{aligned}
  			\mathfrak{g}_{1\bar 1 i\bar i} \geq \,&
  			\mathfrak{g}_{i\bar i 1\bar 1}+
  			2\mathfrak{Re} (\bar T^j_{1i}\mathfrak{g}_{1\bar j i})
  			-C\sqrt{Q} \nonumber
  		\end{aligned}
  	\end{equation}
  	where 
  	$Q=|\partial\overline{\partial}u|^2+|\partial \partial u|^2$, as defined above.
  	Differentiating the equation 
  	\eqref{main-equ2}
  	we obtain
  	\begin{equation}
  		\label{gqy-45diff the equation}
  		\begin{aligned}
  			F^{i\bar i}\mathfrak{g}_{i\bar i l}  =\psi_{z_l}+\psi_uu_l,  \nonumber
  		\end{aligned}
  	\end{equation}
  	\begin{equation}
  		\begin{aligned}
  			F^{i\bar i}\mathfrak{g}_{i\bar i 1\bar 1}
  			=\psi_{z_1\bar z_1}+\psi_uu_{1\bar 1}+ 
  			2\mathfrak{Re}(\psi_{z_1u}u_1)
  			+\psi_{uu}|u_1|^2
  			-F^{i\bar j,l\bar m }\mathfrak{g}_{i\bar j1}\mathfrak{g}_{l\bar m\bar 1}.     \nonumber
  		\end{aligned}
  	\end{equation}
  	
  	Putting the above inequalities
  	into \eqref{gblq-C90} we get
  	\begin{equation}
  		\label{gblq-C90-2}
  		\begin{aligned}
  			0  \geq  
  			\mathfrak{g}_{1\bar 1} \mathcal{L}\phi -\mathfrak{g}_{1\bar 1}F^{i\bar i}\phi_i\phi_{\bar i}
  			-2\mathfrak{g}_{1\bar 1}\mathfrak{Re} F^{i\bar i}\bar T^1_{1i}\phi_{i}
  			-C\sqrt{Q}\sum F^{i\bar i}
  			+\psi_uu_{1\bar 1}.  \nonumber
  		\end{aligned}
  	\end{equation}

  	Let $\phi=\log\eta+\varphi(w)$, where $w = |\partial u|^2$ is as above,  $\eta$ is the cutoff function given by  \eqref{2-22} and 
  	$$\varphi=\varphi(w)=(1-\frac{w}{2N})^{-\frac{1}{2}} \mbox{ where   $N=\sup_{\{\eta>0\}}|\partial u|^2$}. $$
  	Then
  	\begin{equation}
  		\begin{aligned}
  			\mathcal{L}\phi=\,&
  			\frac{\mathcal{L}\eta}{\eta} -F^{i\bar i}\frac{|\eta_i|^2}{\eta^2}+\varphi' \mathcal{L}w +\varphi''F^{i\bar i}|w_i|^2,  \nonumber
  		\end{aligned}
  	\end{equation}
  	\begin{equation}
  		\begin{aligned}
  			F^{i\bar i}|\phi_i|^2
  			+2 \mathfrak{Re} F^{i\bar i}\bar T^1_{1i}\phi_{i} \leq \frac{4}{3}F^{i\bar i}|\phi_i|^2+C\sum F^{i\bar i},  \nonumber
  		\end{aligned}
  	\end{equation}
  	\begin{equation}
  		\begin{aligned}
  			F^{i\bar i} |\phi_i|^2
  			\leq \frac{3}{2}F^{i\bar i}|\varphi_i|^2+3F^{i\bar i}\frac{|\eta_i|^2}{\eta^2}.  \nonumber
  		\end{aligned}
  	\end{equation}
Moreover, one can check $\varphi'=\frac{\varphi^3}{4N},$ $\varphi''=\frac{3\varphi^5}{16N^2}$ and $1\leq \varphi\leq\sqrt{2}$.
  	And so
  	\begin{equation}
  		\begin{aligned}
  			\varphi''-2\varphi'^2=\frac{\varphi^5}{16N^2}(3-2\varphi)>\frac{\varphi^5}{96N^2}.  \nonumber
  		\end{aligned}
  	\end{equation}
  	By Lemma \ref{lemma-Lw} 
  	\begin{equation}
  		\begin{aligned}
  			\mathcal{L}(w) \geq   \frac{3\theta Q}{4} \sum F^{ii}
  			-C \left(1+\sum F^{i\bar i} \right).   \nonumber
  		\end{aligned}
  	\end{equation}
  	And \eqref{2-22} tells us that
  	\begin{equation}
  		\begin{aligned}
  			0\leq \eta \leq 1, ~~\eta|_{B_{\frac{r}{2}}}\equiv 1,
  			~~\frac{F^{i\bar i}|\eta_i|^2}{\eta^2} \leq \frac{C}{r^2 \eta},  
  			~~\frac{\mathcal{L}\eta}{\eta} \leq  \frac{C}{r^2\eta} \sum F^{i\bar i}. \nonumber
  		\end{aligned}
  	\end{equation}
  	
  	In conclusion we finally obtain
  	\begin{equation}
  		\begin{aligned}
  			0	\geq
  			\frac{9\theta Q}{16N} \sum F^{i\bar i}-\frac{C}{r^2\eta}
  			\sum F^{i\bar i}
  			-C\frac{\sqrt{Q}}{\mathfrak{g}_{1\bar{1}}} \sum F^{i\bar i}
  			-\frac{\psi_u\chi_{1\bar 1}}{\mathfrak{g}_{1\bar 1}}
  			+\psi_u.   \nonumber
  		\end{aligned}
  	\end{equation}
  	Combining  \eqref{sum-2}, we obtain
  $\eta \mathfrak{g}_{1\bar 1}\leq C.$


   	\bigskip
  	
  \medskip

  	\begin{appendix}
  		

  		\section{Proof of Lemmas 
  			\ref{yuan-k+1}, \ref{lemma5.11}, \ref{lemma1-unbound-type2} and  \ref{lemma23}}
  		\label{appendix2}

  		Fix $\lambda\in\Gamma$ with
  		$\lambda_1 \leq \cdots \leq \lambda_n$. 
  		The concavity and symmetry of $f$ yields that
  		\begin{equation}
  			\label{concavity-symmetry2}
  			f_1(\lambda)\geq \cdots\geq f_n(\lambda), 
  			\, f_1(\lambda)\geq \frac{1}{n}\sum f_i(\lambda).
  		\end{equation}
  		When $\Gamma=\Gamma_n$ (if and only if $\kappa_\Gamma=0$), we obtain Lemma \ref{yuan-k+1}. 
  		When $\Gamma\neq\Gamma_n$,
  		Lemmas \ref{yuan-k+1} and \ref{lemma5.11} are consequences of the following two propositions.

  		\begin{proposition}[\cite{yuan2020conformal,yuan-PUE1}]
  			\label{yuanrr-2}
  			Assume $\Gamma\neq \Gamma_n$ and  $f$ satisfies  
  			\eqref{addistruc} in $\Gamma$. Let 
  			$\kappa_\Gamma$ be as  
  			in \eqref{def1-kappa-gamma}. 
  			Let
  			$\alpha_1, \cdots, \alpha_n$ be 
  			positive constants with
  	$(-\alpha_1,\cdots,-\alpha_{\kappa_\Gamma}, \alpha_{\kappa_\Gamma+1},\cdots, \alpha_n)\in \Gamma. $
  			In addition,  assume $\alpha_1\geq \cdots\geq  \alpha_{\kappa_\Gamma}$.
  			Then for any $ \lambda\in \Gamma$ with order $\lambda_1 \leq  \cdots \leq \lambda_n$,
  			\begin{equation}
  				\label{theta1}
  				\begin{aligned}
  					f_{\kappa_\Gamma+1}(\lambda)\geq \frac{\alpha_1}{\sum_{i=\kappa_\Gamma+1}^n \alpha_i-\sum_{i=2}^{\kappa_\Gamma}\alpha_i}f_1(\lambda).
  				\end{aligned}
  			\end{equation}
  		\end{proposition}
  		
  		
  		\begin{proof}
  			 By  
  	  Lemma \ref{lemma3.4}  we  conclude that 
  			 	$f_i(\lambda)\geq 0$, $\sum  f_i(\lambda)>0$ and 
  			 	$	-\sum_{i=1}^{\kappa_\Gamma} \alpha_i f_i(\lambda)+\sum_{i=\kappa_\Gamma+1}^n \alpha_i f_i(\lambda)\geq0.$
  		This yields $f_{\kappa_\Gamma+1}(\lambda)\geq   \frac{\alpha_1}{\sum_{i=\kappa_\Gamma+1}^n \alpha_i}f_1(\lambda)$. 
  	 Moreover,
  		 one can derive \eqref{theta1} 	by iteration.
  			
  		\end{proof}
  		


  		\begin{proposition}[\cite{yuan2020conformal,yuan-PUE1}]
  			\label{thm-k+1}
  			
  			We assume  that
  			$f$ is of $(k+1)$-uniform ellipticity in the corresponding cone $\Gamma$ for some $1\leq k\leq  n-1$.  Then $\kappa_\Gamma\geq k$.
  			
  		\end{proposition}

  		\begin{proof}

  			
  		 	Let $\vartheta$ be as in \eqref{partial-uniform2}.
  			Let   
  			$c_0>0$ be some constant with
  			$f(c_0 \vec{\bf 1})>\underset{\partial\Gamma}{\sup}f$.
  			Take $a=1+c_0$  then $f(a\vec{\bf 1})>f(c_0 \vec{\bf1})$. 
  			For 
  			$\epsilon>0$ 
  			and 
  			$R>0$, we denote
  			$\lambda_{\epsilon,R}=({\overbrace{\epsilon,\cdots,\epsilon}^{k}},{\overbrace{R,\cdots, R}^{n-k}}).$
  			We can deduce from 
  			\eqref{concavity2} that
  			\begin{equation}
  				\begin{aligned}
  					f(\lambda_{\epsilon,R})\geq 
  					\,& f(a\vec{\bf 1})+\epsilon\sum_{i=1}^{k} f_i(\lambda_{\epsilon,R})+
  					R\sum_{i=k+1}^n  f_i(\lambda_{\epsilon,R})
  					-a\sum_{i=1}^n f_i(\lambda_{\epsilon,R}) 
  					\\
  					\geq \,& f(a\vec{\bf 1})+(R\vartheta -a)\sum_{i=1}^{n} f_i(\lambda_{\epsilon,R}) \mbox{ (using $(k+1)$-uniform ellipticity)}
  					\\
  					= \,& f(a\vec{\bf 1}) \mbox{ (by setting } R=\frac{a}{\vartheta}).
  					\nonumber
  				\end{aligned}
  			\end{equation}
  			Notice  $R=\frac{a}{\vartheta}$  depends not on $\epsilon$. 
  			Next we prove 
  			$\lambda_{\epsilon,R}\to\lambda_{0,R}=(0,\cdots,0,R,\cdots,R)\in  \Gamma$ as $\epsilon\rightarrow 0^+$. 
  			If $\lambda_{0,R}\in\partial\Gamma$ then 		
  			${\sup}_{\partial\Gamma}f \geq \lim_{\epsilon\to 0^+}  f(\lambda_{\epsilon,R})\geq f(a\vec{\bf1})$. 
  			A contradiction.
  		\end{proof}

  		\begin{proof}
  			[Proof of Lemma \ref{yuan-k+1}]
  			Fix $\lambda\in\Gamma$ with $\lambda_1\leq \cdots \leq \lambda_n.$	
  			By 
  			 Lemma \ref{lemma3.4}, we know
  			$f_i(\lambda)\geq 0, \, \sum f_i(\lambda)>0$, thereby confirming    $(1)$. 
  			
  			Next, we will verify $(2)$.
  			When $\Gamma=\Gamma_n$ this is trivial. 
  			The remaining case $\Gamma\neq\Gamma_n$ follows immediately from   Proposition \ref{yuanrr-2} and \eqref{concavity-symmetry2}.
  			
  			Finally, we will prove that the statement of   $(\kappa_\Gamma+1)$-uniform ellipticity  is sharp. Assume by contradiction that  $f$ is of $(k+1)$-uniform ellipticity for some $k>\kappa_\Gamma$. Then  $\kappa_\Gamma\geq k$ according to Proposition \ref{thm-k+1}. This is a  contradiction.
  			
  		\end{proof}

  		\begin{proof}
  			[Proof of Lemma  \ref{lemma5.11}] 
  			Obviously $(1) \Rightarrow (2)$.  
  			Since $\Gamma$ is open,   $(2) \Rightarrow (1)$. By Lemma  \ref{yuan-k+1}, $(2) \Rightarrow (3)$.   
  		By Proposition \ref{thm-k+1},  $\kappa_\Gamma\geq n-1$. Thus
  		$(3) \Rightarrow (2)$.

  		\end{proof}

  		\begin{proof}
  			[Proof of Lemma \ref{lemma1-unbound-type2}]
  			Fix $\lambda\in\Gamma$. For $t>0$, we denote $\lambda^t=(\lambda_1,\cdots,\lambda_{n-1}\lambda_n+t).$ By Lemma \ref{lemma5.11}, $(0,\cdots,0,1)\in\Gamma$. Using  \eqref{uniform-elliptic-2}, for large $t$ 
  			\begin{align*}
  				f(\lambda^t)-f(\frac{t}{2}(0,\cdots,0,1))
  				\geq \sum_{i=1}^n f_i(\lambda^t)\lambda_i +\frac{t}{2} f_n(\lambda^t)\geq 0.
  			\end{align*}
  			Together with \eqref{addistruc}, we know that $f$ satisfies the unbounded condition	\eqref{unbounded-1}.
  		\end{proof}

  		\begin{proof}
  			[Proof of Lemma \ref{lemma23}]
  			
  		Fix $\lambda \in\Gamma$.
  			Note that   \eqref{homogeneous-1-buchong2} and  the concavity  imply 
  			$ f_i(\lambda)\geq0$. Combining   $\sup_\Gamma f=+\infty$, we get $\sum_{i=1}^n f_i(\lambda)>0$  (otherwise $f(\mu)\leq f(\lambda)$,   $\forall\mu\in\Gamma$).  
  			Fix   $R>0$. Then $t\lambda -R\vec{\bf1}\in\Gamma$ for some $t>0.$
  			By concavity and \eqref{homogeneous-1-buchong2},  we have
  			\begin{align*}
  				f(t\lambda) \geq sf(\frac{R}{s}\vec{\bf 1}) +(1-s) f(\frac{t\lambda-R\vec{\bf1}}{1-s})>sf(\frac{R}{s}\vec{\bf1}), \, \forall 0<s<1.
  			\end{align*}
  			So $f(t\lambda)\geq f(R\vec{\bf1})$ for $t\gg1$.
  		\end{proof}

 \medskip
  		\section{Proof of Lemma \ref{yuan's-quantitative-lemma}} 
  		\label{appendix1}


  		
  		For convenience we give the proof of Lemma \ref{yuan's-quantitative-lemma} in this appendix. 
  		We start with the case $n=2$. 
  		For $n=2$, the eigenvalues of $\mathrm{A}$ are
  		$$\lambda_{1}=\frac{\mathrm{{\bf a}}+d_1- \sqrt{(\mathrm{{\bf a}}-d_1)^2+4|a_1|^2}}{2} 
  		\mbox{ and } \lambda_2=\frac{\mathrm{{\bf a}}+d_1+\sqrt{(\mathrm{{\bf a}}-d_1)^2+4|a_1|^2}}{2}.$$
  		We can assume $a_1\neq 0$; otherwise we are done.
  		If $\mathrm{{\bf a}} \geq \frac{|a_1|^2}{ \epsilon}+ d_1$ then one has
  		\begin{equation}
  			\begin{aligned}
  				0\leq d_1- \lambda_1 =\lambda_2-\mathrm{{\bf a}}
  				= \frac{2|a_1|^2}{\sqrt{ (\mathrm{{\bf a}}-d_1)^2+4|a_1|^2 } +(\mathrm{{\bf a}}-d_1)}
  				< \frac{|a_1|^2}{\mathrm{{\bf a}}-d_1 } \leq \epsilon.   \nonumber
  			\end{aligned}
  		\end{equation}

  		The following lemma enables us  to count  the eigenvalues near the diagonal elements
  		via a deformation argument.
  		It is an essential  ingredient in the proof of  Lemma \ref{yuan's-quantitative-lemma}   for general $n$.
  		\begin{lemma}
  			[\cite{yuan2017,yuan-regular-DP}]
  			\label{refinement}
  			Let $\mathrm{A}$ be an $n\times n$  Hermitian matrix
  			\begin{equation}
  				\label{matrix2}
  				\left(
  				\begin{matrix}
  					d_1&&  &&a_{1}\\
  					&d_2&& &a_2\\
  					&&\ddots&&\vdots \\
  					&& &  d_{n-1}& a_{n-1}\\
  					\bar a_1&\bar a_2&\cdots& \bar a_{n-1}& \mathrm{{\bf a}} \nonumber
  				\end{matrix}
  				\right)
  			\end{equation}
  			with $d_1,\cdots, d_{n-1}, a_1,\cdots, a_{n-1}$ fixed, and with $\mathrm{{\bf a}}$ variable.
  			Denote
  			$\lambda_1,\cdots, \lambda_n$ by the eigenvalues of $\mathrm{A}$ with the order
  			$\lambda_1\leq \lambda_2 \leq\cdots \leq \lambda_n$.
  			Fix  $\epsilon>0$.
  			Suppose that the parameter $\mathrm{{\bf a}}$ in the matrix $\mathrm{A}$ 
  			satisfies  the following quadratic growth condition
  			\begin{equation}
  				\label{guanjian2}
  				\begin{aligned}
  					\mathrm{{\bf a}} \geq \frac{1}{\epsilon}\sum_{i=1}^{n-1} |a_i|^2+\sum_{i=1}^{n-1}  [d_i+ (n-2) |d_i|]+ (n-2)\epsilon.
  				\end{aligned}
  			\end{equation}
  			Then for any $\lambda_{\alpha}$ $(1\leq \alpha\leq n-1)$ there exists $d_{i_{\alpha}}$
  			with 
  			index $1\leq i_{\alpha}\leq n-1$ such that
  			\begin{equation}
  				\label{meishi}
  				\begin{aligned}
  					|\lambda_{\alpha}-d_{i_{\alpha}}|<\epsilon,
  				\end{aligned}
  			\end{equation}
  			\begin{equation}
  				\label{mei-23-shi}
  				0\leq \lambda_{n}-\mathrm{{\bf a}} <(n-1)\epsilon + |\sum_{\alpha=1}^{n-1}(d_{\alpha}-d_{i_{\alpha}})|.
  			\end{equation}
  		\end{lemma}

  		\begin{proof}
  			Without loss of generality, we assume $\sum_{i=1}^{n-1} |a_i|^2>0$ and  $n\geq 3$  (otherwise we are done).
  			Note that in the assumption of the lemma the eigenvalues have
  			the order $\lambda_1\leq \lambda_2\leq \cdots \leq \lambda_n$.
  			It is  well known that, for a Hermitian matrix,
  			any diagonal element is   less than or equals to   the  largest eigenvalue.
  			In particular,
  			\begin{equation}
  				\label{largest-eigen1}
  				\lambda_n \geq \mathrm{{\bf a}}.
  			\end{equation}
  			
  			It 
  			suffices to prove \eqref {meishi}, since  \eqref{mei-23-shi} is a consequence of  \eqref{meishi}, \eqref{largest-eigen1}  and
  			\begin{equation}
  				\label{trace}
  				\sum_{i=1}^{n}\lambda_i=\mbox{tr}(\mathrm{A})=\sum_{\alpha=1}^{n-1} d_{\alpha}+\mathrm{{\bf a}}.
  			\end{equation}

  			Let's denote   $I=\{1,2,\cdots, n-1\}$. We divide the index set   $I$ into two subsets  by
  			$${\bf B}=\{\alpha\in I: |\lambda_{\alpha}-d_{i}|\geq \epsilon, \mbox{   }\forall i\in I\} $$
  			and $ {\bf G}=I\setminus {\bf B}=\{\alpha\in I: \mbox{There exists an $i\in I$ such that }
  			|\lambda_{\alpha}-d_{i}| <\epsilon\}.$
  			
  			To complete the proof we need to prove ${\bf G}=I$ or equivalently ${\bf B}=\emptyset$.
  			It is easy to see that  for any $\alpha\in {\bf G}$, one has
  			\begin{equation}
  				\label{yuan-lemma-proof1}
  				\begin{aligned}
  					|\lambda_\alpha|< \sum_{i=1}^{n-1}|d_i| + \epsilon.
  				\end{aligned}
  			\end{equation}
  			
  			Fix $ \alpha\in {\bf B}$,  we are going to 
  			estimate $\lambda_\alpha$.
  			The eigenvalue $\lambda_\alpha$ satisfies
  			\begin{equation}
  				\label{characteristicpolynomial}
  				\begin{aligned}
  					(\lambda_{\alpha} -\mathrm{{\bf a}})\prod_{i=1}^{n-1} (\lambda_{\alpha}-d_i)
  					= \sum_{i=1}^{n-1} (|a_{i}|^2 \prod_{j\neq i} (\lambda_{\alpha}-d_{j})).
  				\end{aligned}
  			\end{equation}
  			By the definition of ${\bf B}$, for  $\alpha\in {\bf B}$, one then has $|\lambda_{\alpha}-d_i|\geq \epsilon$ for  $i\in I$.
  			We   derive
  			\begin{equation}
  				\begin{aligned}
  					|\lambda_{\alpha}-\mathrm{{\bf a}} |=  \left|\sum_{i=1}^{n-1} \frac{|a_i|^2}{\lambda_{\alpha}-d_{i}}\right|\leq\sum_{i=1}^{n-1} \frac{|a_i|^2}{|\lambda_{\alpha}-d_{i}|}\leq
  					\frac{1}{\epsilon}\sum_{i=1}^{n-1} |a_i|^2, \mbox{ if } \alpha\in {\bf B}.
  				\end{aligned}
  			\end{equation}
  			Hence,  for $\alpha\in {\bf B}$, we obtain
  			\begin{equation}
  				\label{yuan-lemma-proof2}
  				\begin{aligned}
  					\lambda_\alpha \geq \mathrm{{\bf a}}-\frac{1}{\epsilon}\sum_{i=1}^{n-1} |a_i|^2.
  				\end{aligned}
  			\end{equation}

  			We shall use proof by contradiction to prove  ${\bf B}=\emptyset$.
  			For a set ${\bf S}$, we denote $|{\bf S}|$ the  cardinality of ${\bf S}$.
  			Assume ${\bf B}\neq \emptyset$.
  			Then $|{\bf B}|\geq 1$, and so $|{\bf G}|=n-1-|{\bf B}|\leq n-2$. 

  			In the case  ${\bf G}=\emptyset$, one knows that
  			\begin{equation}
  				\begin{aligned}
  					\mbox{tr}(\mathrm{A})
  					\geq
  					\mathrm{{\bf a}}+
  					(n-1) (\mathrm{{\bf a}}-\frac{1}{\epsilon}\sum_{i=1}^{n-1} |a_i|^2 )
  					>   \sum_{i=1}^{n-1}d_i +\mathrm{{\bf a}}= \mbox{tr}(\mathrm{A}).
  				\end{aligned}
  			\end{equation}

  			In the case  ${\bf G}\neq \emptyset$, we compute the trace of the matrix $A$ as follows:
  			\begin{equation}
  				\begin{aligned}
  					\mbox{tr}(\mathrm{A})= \,&
  					\lambda_n+
  					\sum_{\alpha\in {\bf B}}\lambda_{\alpha} + \sum_{\alpha\in  {\bf G}}\lambda_{\alpha}  \\
  					\geq \,&
  					\lambda_n+
  					|{\bf B}| (\mathrm{{\bf a}}-\frac{1}{\epsilon}\sum_{i=1}^{n-1} |a_i|^2 )-|{\bf G}| (\sum_{i=1}^{n-1}|d_i|+\epsilon ) \\
  					> \,& 
  					2\mathrm{{\bf a}}-\frac{1}{\epsilon}\sum_{i=1}^{n-1} |a_i|^2 -(n-2) (\sum_{i=1}^{n-1}|d_i|+\epsilon )
  					\\
  					\geq \,& \sum_{i=1}^{n-1}d_i +\mathrm{{\bf a}}= \mbox{tr}(\mathrm{A}),
  				\end{aligned}
  			\end{equation}
  			where we use  \eqref{guanjian2},   \eqref{largest-eigen1}, \eqref{yuan-lemma-proof1} and \eqref{yuan-lemma-proof2}. 
  			Again, it is a contradiction.  Thus  ${\bf B}=\emptyset$ as required. 
  			
  		\end{proof}
  		
  		We apply Lemma \ref{refinement} to prove Lemma \ref{yuan's-quantitative-lemma} via a deformation argument.

  		\begin{proof}
  			[Proof of Lemma \ref{yuan's-quantitative-lemma}]
  			Without loss of generality,  we assume $n\geq 3$,  $\sum_{i=1}^{n-1} |a_i|^2>0$.
  			Fix $a_1, \cdots, a_{n-1}$,
  			$d_1, \cdots, d_{n-1}$. 
  			Denote $\lambda_1(\mathrm{{\bf a}}), \cdots, \lambda_n(\mathrm{{\bf a}})$  
  			the eigenvalues of $\mathrm{A}$ 
  			with
  			$$\lambda_1(\mathrm{{\bf a}})\leq \cdots\leq \lambda_n(\mathrm{{\bf a}}).$$ 
  			Clearly,  the eigenvalues $\lambda_i(\mathrm{{\bf a}})$ are all continuous functions 
  			in $\mathrm{{\bf a}}$.
  			
  			For simplicity, we write $\lambda_i=\lambda_i(\mathrm{{\bf a}})$. 
  			%
  			Fix $\epsilon>0$. 
  			Let $I'_\alpha=(d_\alpha-\frac{\epsilon}{2n-3}, d_\alpha+\frac{\epsilon}{2n-3})$,  
  			$$P_0'=\frac{2n-3}{\epsilon}\sum_{i=1}^{n-1} |a_i|^2+ (n-1)\sum_{i=1}^{n-1} |d_i|+ \frac{(n-2)\epsilon}{2n-3}.$$
  			In what follows we assume  $\mathrm{{\bf a}}\geq P_0'$.
  			The connected components of $\bigcup_{\alpha=1}^{n-1} I_{\alpha}'$ are as in the following:
  			$$J_{1}=\bigcup_{\alpha=1}^{j_1} I_\alpha',
  			J_2=\bigcup_{\alpha=j_1+1}^{j_2} I_\alpha'  \cdots, J_i =\bigcup_{\alpha=j_{i-1}+1}^{j_i} I_\alpha', \cdots, 
  			J_{m} =\bigcup_{\alpha=j_{m-1}+1}^{n-1} I_\alpha'.$$
  			Moreover, $J_i\bigcap J_k=\emptyset, \mbox{ for }   1\leq i<k\leq m. $
  			%
  			Let  $ \mathrm{{\bf \widetilde{Card}}}_k:[P_0',+\infty)\rightarrow \mathbb{N}$
  			be the function that counts the eigenvalues which lie in $J_k$.
  			(Note that when the eigenvalues are not distinct,  the function $\mathrm{{\bf \widetilde{Card}}}_k$ denotes  the summation of all the algebraic  multiplicities of  distinct eigenvalues which
  			lie in $J_k$).
  			This function measures the number of the  eigenvalues which lie in $J_k$.
  			The crucial ingredient is that  Lemma \ref{refinement}  yields the continuity of   $\mathrm{{\bf \widetilde{Card}}}_i(\mathrm{{\bf a}})$ for $\mathrm{{\bf a}}\geq P_0'$. More explicitly,
  			by Lemma \ref{refinement} and  $\lambda_n \geq {\bf a}$
  			we conclude that,
  			if   $\mathrm{{\bf a}}$ satisfies the quadratic growth condition \eqref{guanjian1-yuan} then
  			\begin{equation}
  				\label{yuan-lemma-proof5}
  				\begin{aligned}
  					\lambda_n \in \mathbb{R}\setminus (\bigcup_{k=1}^{n-1} \overline{I_k'})
  					=\mathbb{R}\setminus (\bigcup_{i=1}^m \overline{J_i}), 
  					\mbox{ and }
  					\lambda_\alpha \in \bigcup_{i=1}^{n-1} I_{i}'=\bigcup_{i=1}^m J_{i} \mbox{ for } 1\leq\alpha\leq n-1.  \nonumber
  				\end{aligned}
  			\end{equation}
  			Hence,  $\mathrm{{\bf \widetilde{Card}}}_i(\mathrm{{\bf a}})$ is a continuous function
  			in the variable $\mathrm{{\bf a}}$. So it is a constant.
  			Together with  the line of the proof   of \cite[Lemma 1.2]{CNS3}
  			we see
  			that $ \mathrm{{\bf \widetilde{Card}}}_i(\mathrm{{\bf a}}) =j_i-j_{i-1}$ for sufficiently large $\mathrm{{\bf a}}$.
  			Here we denote $j_0=0$ and $j_m=n-1$.
  			The constant of $ \mathrm{{\bf \widetilde{Card}}}_i$  therefore follows that
  			$$ \mathrm{{\bf \widetilde{Card}}}_i(\mathrm{{\bf a}})
  			=j_i-j_{i-1}.$$
  			We thus know that the   $(j_i-j_{i-1})$ eigenvalues
  			$$\lambda_{j_{i-1}+1}, \lambda_{j_{i-1}+2}, \cdots, \lambda_{j_i}$$
  			lie in the connected component $J_{i}$.
  			Thus, for any $j_{i-1}+1\leq \gamma \leq j_i$,  we have $I_\gamma'\subset J_i$ and  $\lambda_\gamma$
  			lies in the connected component $J_{i}$.
  			Therefore,
  			$$|\lambda_\gamma-d_\gamma| < \frac{(2(j_i-j_{i-1})-1) \epsilon}{2n-3}\leq \epsilon.$$
  			Here we  use the fact that $d_\gamma$ is midpoint of  $I_\gamma'$ and 
  			every $J_i\subset \mathbb{R}$ is an open subset.
  			
  			To be brief,  if for fixed index $1\leq i\leq n-1$ the eigenvalue $\lambda_i(P_0')$ lies in $J_{\alpha}$ for some $\alpha$, 
  			then  Lemma \ref{refinement} implies that, for any ${\bf a}>P_0'$, the corresponding eigenvalue  $\lambda_i({\bf a})$ lies in the same  interval $J_{\alpha}$.
  			The computation of $\mathrm{{\bf \widetilde{Card}}}_k$ can be done by setting $\mathrm{\bf a}\rightarrow+\infty$.
  			
  			
  		\end{proof}

  	\end{appendix}


   \medskip
  \noindent{\bf Acknowledgements}.
  The author is partially supported by  
Guangzhou Science and Technology Program  (Grant No. 202201010451) and Guangdong Basic and Applied Basic Research Foundation (Grant No. 2023A1515012121).

\bigskip


\end{document}